\documentclass{article}
\usepackage[utf8]{inputenc}
\usepackage[english]{babel}
\usepackage{caption}
\usepackage{authblk}
\usepackage{subcaption}
\usepackage{amsthm,amssymb,amsmath,amsfonts,calrsfs,upgreek,subcaption,caption,pst-plot,enumitem,stmaryrd}
\usepackage[mathscr]{euscript}
\usepackage{mathrsfs}
\newtheorem{The}{Theorem}[section]

\newtheorem{Def}{Definition}[section]
\newtheorem{lem}{Lemma}[section]
\newtheorem{cor}{Corollary}[section]
\newtheorem{Pro}{Proposition}[section]
\newtheorem{conj}{Conjecture}[section]
\newtheorem{que}{Problem}[section]
\newcommand{\La}{\mathscr{L}}
\newcommand{\I}{\mathscr{I}}
\newcommand{\F}{\mathscr{F}}

\newcommand{\D}{\mathscr{D}}

\newcommand{\oo}{\mathbf{0}}

\newcommand{\q}{\mathbf{q}}
\newcommand{\uu}{\mathbf{u}}
\newcommand{\rr}{\mathbf{r}}
\newcommand{\vv}{\mathbf{v}}

\newcommand{\x}{\mathbf{x}}
\newcommand{\ii}{\mathbf{i}}
\newcommand{\y}{\mathbf{y}}
\newcommand{\z}{\mathbf{z}}

\newcommand{\tv}{\mathbf{t}}
\newcommand{\s}{\mathbf{s}}

\newcommand{\A}{\mathscr{A}}

\newcommand{\e}{\epsilon}
\newcommand{\E}{\mathcal{E}}
\newcommand{\Z}{\mathbb{Z}}
\newcommand{\Q}{\mathbb{Q}}
\newcommand{\N}{\mathbb{N}}
\newcommand{\R}{\mathbb{R}}
\newcommand{\Co}{\mathbb{C}}
\newcommand{\relog}{\textrm{Relog}}
\newcommand{\norm}[1]{\left\lVert#1\right\rVert_1}\usepackage{tikz}\usetikzlibrary{automata, positioning, arrows, decorations.pathreplacing,angles,quotes}\tikzset{ 
    >=stealth, 
    node distance=3cm, 
    every state/.style={thick, fill=gray!10}, 
    every edge/.append style={line width=0.25mm}, 
    initial text=$ $, 
    }
    
\title{Multivariate growth and cogrowth}
\author[1]{Rostislav Grigorchuk}
\author[2]{Jean-Fran\c{c}ois Quint}
\author[3]{Asif Shaikh}
\affil[1]{Department of Mathematics, Texas A\&M University, College Station, TX, USA,
grigorch@math.tamu.edu}
\affil[2]{CNRS—Université Bordeaux I, 33405, Talence, France, jquint@math.u-bordeaux.fr}
\affil[3]{Department of Mathematics \& Statistics, R. A. Podar College of Commerce and Economics, Mumbai, India, asif.shaikh@rapodar.ac.in}
\date{ }
\usepackage{graphicx}
\begin{document}
\maketitle
\begin{abstract}
We investigate  a multivariate growth series $\Gamma_L(\z), \z \in \Co^d$ associated with a regular language $L$ over an alphabet of cardinality $d.$ Our focus is on languages coming from subgroups of the free group and from  subshifts  of  finite  type. We develop a mechanism for computing the rate of growth $\varphi_L(\rr)$ of $L$ in the direction $\rr \in \R^d$. Using the concave growth condition (CG) introduced by the second author in  \cite{quint2002divergence} and the results of Convex Analysis we represent $\psi_L(\rr) = \log\left(\varphi_L(\rr)\right)$ as a support function of a convex set that is a closure of the $\relog$ image  of  the  domain of absolute convergence of $\Gamma_L(\z)$. This allows us to compute $\psi_L(\rr)$ in some important cases, like a Fibonacci language or a language of freely reduced words representing elements of a free group $F_2$. Also we show that the methods of the Large deviation theory can be used as an alternative approach. Finally, we suggest some open problems directed on the possibility of extensions of the results of the first author from  \cite{grigorchuk1980symmetrical}  on multivariate  cogrowth. 
\end{abstract}
\section{Introduction}\label{sec:intro}
The study of asymptotic properties of a sequence $\{\gamma_n\}_{n\geq 0}$ of real (or complex) numbers is related to study of analytic properties of the function $\Gamma(z)$ presented by a power series $\displaystyle\sum_{n=0}^{\infty} \gamma_n z^n.$ This includes inspection of singularities on the border of the domain of (absolute) convergence of $\Gamma(z)$ and local behavior of $\Gamma(z)$ in their neighborhood. If $\Gamma(z)$ is a rational function i.e. $$\Gamma(z) = \displaystyle\frac{G(z)}{H(z)}, ~~G(z), H(z) \in \Co[z]$$ then all needed information about the coefficients $\gamma_n$ can be gained from the polynomials $G(z)$ and $H(z).$\\

In algebra, and especially in modern group theory, there are many notions and concepts that attached to the algebraic object a sequence $\{\gamma_n\}_{n\geq 0}.$ This include growth, cogrowth, subgroup growth etc. Recall that given a finitely generated group $G$ with a system of generators $S,$ one can consider a function
$$\gamma_n = \#\{g \in G : |g| = n\},$$ where $n \in \N$ and $|g|$ is the length of the element $g$ with respect to $S.$ If the pair $(G,S)$ has a regular geodesic normal form (in other terminology a rational cross section \cite{MR895616gilman}) then the power series 
$$\Gamma(z) = \displaystyle\sum_{n=0}^{\infty} \gamma_n z^n$$ represents a rational function and the asymptotic of $\gamma_n$ is either polynomial or exponential. There are many groups (for instance groups of intermediate growth constructed in \cite{gri83milnor}) for which $\Gamma(z)$ is irrational for any system of generators and the 
study of asymptotic properties of $\{\gamma_n\}_{n=0}^{\infty}$ becomes much more complicated. \\

Now let $F_m$ be a free group of rank $m.$ Every $m$ generated group $G$ can be presented as a  quotient $F_m/N,$ for a suitable normal subgroup $N \vartriangleleft F_m.$ Let $A = \{a_1,\cdots,a_m\}$ be a basis of $F_m$. Elements of $F_m$ are presented by freely reduced words over alphabet $\Sigma = \{a_1,\cdots,a_m,a_1^{-1},\cdots,a_m^{-1}\}$ and there is $2m(2m-1)^{n-1}$ such words of length $n\geq 1$. The function
$$\Gamma_{F_m}(z) = \displaystyle\sum_{n=0}^{\infty} 2m(2m-1)^{n-1} z^n = \displaystyle \frac{1+z}{1-(2m-1)z}$$
is a spherical growth function of $F_m$ with respect to the basis $A$. 
Now let $H < F_m$ be a subgroup and $H_n$ be the set of elements in $H$ of length $n$ with respect to generators $\{a_1,\cdots,a_m\}$ of $F_m.$ The sequence $\{|H_n|\}_n^{\infty}$ of cardinalities of these sets is a cogrowth sequence,
$$H(z) = \displaystyle \sum_{n=0}^{\infty} |H_n| z^n$$
is a cogrowth series, and 
$$\alpha_H = \displaystyle \limsup_{n \rightarrow \infty} |H_n|^{\frac{1}{n}}$$ is a cogrowth. The range for $\alpha_H$ is $[1,2m-1]$ and the range of $\alpha_H$ when $H$ is nontrivial and normal subgroup is $\left(\sqrt{2m-1},2m-1\right].$ (See \cite{griMR0474511_1977,gri1978thesis,gri_MR552478_1979,grigorchuk1980symmetrical}). The spectral radius $\chi$ of the simple random walk on $G = F_m/N$, when $N \triangleleft F_m$ is related with $\alpha_N$ as
$$\chi = \displaystyle \frac{\sqrt{2m-1}}{2m}\left(\frac{\sqrt{2m-1}}{\alpha_N} + \frac{\alpha_N}{\sqrt{2m-1}}\right)$$
and the group $G$ is amenable if and only if $\alpha_N = 2m-1$ (that is $\alpha_N$ takes its maximum possible value).\\
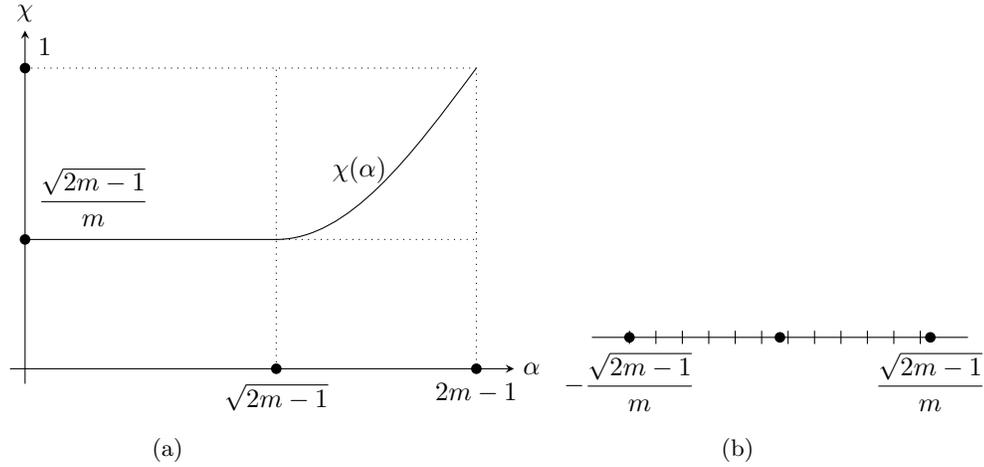
\begin{figure}[!htb]
    \centering
    \begin{subfigure}[b]{.35\textwidth}
    \centering
    \begin{tikzpicture}[domain=0:4]
  \draw[->] (-0.2,0) -- (6.5,0) node[right] {$\alpha$};
  \draw[->] (0,-0.2) -- (0,4.5) node[above] {$\chi$};
  \draw[-] (0,1.72)node[above right=2pt] {$\displaystyle\frac{\sqrt{2m-1}}{m}$} --(3.34,1.72) node[above right=25pt] {$\chi(\alpha)$} cos (6,4);
  \draw[dotted] (3.34,0) node[below=2pt]{$\displaystyle\sqrt{2m-1}$} -- (3.34,4);
  \draw[dotted] (0,4) node[above right=2pt]{$1$} -- (6,4) -- (6,0)node[below=2pt]{$2m-1$} ;
  \draw[dotted]  (3.34,1.72) -- (6,1.72) ;
  \fill (canvas cs:x=3.34cm,y=0cm)    circle (2pt);
  \fill (canvas cs:x=6cm,y=0cm)    circle (2pt);
  \fill (canvas cs:x=0cm,y=1.72cm)    circle (2pt);
  \fill (canvas cs:x=0cm,y=4cm)    circle (2pt);
\end{tikzpicture}
    \caption{}
   \label{subfig:graph_r}
\end{subfigure}
\hfill
\begin{subfigure}[b]{0.4\textwidth}
\centering
    \begin{tikzpicture}[decoration=ticks]
  \draw [decorate] (0,0) -- (4,0);
  \draw[-] (-0.5,0) -- (4.5,0);
  \fill (canvas cs:x=2cm,y=0cm)    circle (2pt);
  \fill (canvas cs:x=0cm,y=0cm)    circle (2pt) node[below=2pt]{$-\displaystyle \frac{\sqrt{2m-1}}{m}$};
  \fill (canvas cs:x=4cm,y=0cm)    circle (2pt) node[below=2pt]{$\displaystyle\frac{\sqrt{2m-1}}{m}$};
\end{tikzpicture}
    \caption{}
    \label{subfig:interval}
\end{subfigure}  
 \caption{The graph of $\chi = \chi(\alpha)$ and the interval $\left[-\displaystyle \frac{\sqrt{2m-1}}{m},\displaystyle \frac{\sqrt{2m-1}}{m}\right]$}
    \label{fig:r}
\end{figure}

In the case when $H < F_m$ is not a normal subgroup, one can consider a Schreier graph $\Lambda = \Lambda(F_m,H,\Sigma).$ Then the dependence on $\alpha_H$ of the spectral radius $\chi$ of a simple random walk on $\Lambda$ is given by
\[ \chi = \left\{
  \begin{array}{lr} 
      \displaystyle \frac{\sqrt{2m-1}}{2m}\left(\frac{\sqrt{2m-1}}{\alpha_H} + \frac{\alpha_H}{\sqrt{2m-1}}\right) & \textnormal{ if } \alpha_H > \sqrt{2m-1} \\
      \displaystyle \frac{\sqrt{2m-1}}{2m} & \textnormal{ if } \alpha_H \leq \sqrt{2m-1}
      \end{array}\right.\]
 (See Figure \ref{subfig:graph_r}). Hence, again the graph $\Lambda$ is amenable if and only if $\alpha_H = 2m-1,$ while in the case when $\Lambda$ is infinite, it is a Ramanujan graph if and only if $\alpha_H \leq 2m-1$. The value $\displaystyle\frac{\sqrt{2m-1}}{m}$ is a spectral radius of random walk on $F_m$ computed by H. Kesten \cite{kesten1959random} and if the spectrum of the Laplacian operator of graph $\Lambda$ minus two points set $\{-1,1\}$ is a subset of the interval given by Figure \ref{subfig:interval}, then the graph is called Ramanujan. Observe that the analogue of the formula for $\chi$ in the context of differential geometry was obtained in \cite{Sullivan_1987_MR882827}.\\
 
 Now we are going to introduce a finer growth characteristics: multivariate growth and multivariate cogrowth.\\
 
 Let $\Sigma = \{a_1,\cdots,a_d\}, d \geq 2$ be an alphabet, $\Sigma^*$ be the set of all finite words (or strings) over $\Sigma.$ The set $\Sigma^*$ with concatenation as a binary operation, can be interpreted as a monoid (with empty word serving as the identity element). In fact, $\Sigma^*$ is a free monoid. Its growth sequence is $d^n, n = 0,1,\cdots.$ Any subset $L \subset \Sigma^*$ is called a (formal) language. With any $w \in \Sigma^*$ we can associate the length $|w|$ and the frequency vector $ \wp(w) = \left(|w|_{a_1},\cdots,|w|_{a_d}\right) \in \N^d$ where $|w|_{a_i}$ is a number of occurrences of the symbol $a_i$ in the word $w.$\\ Let $\z = (z_1,\cdots,z_d) \in \Co^d$ and 

 \begin{equation}
    \Gamma_L(\z) =  \displaystyle \sum_{w \in L} \z^{\wp(w)},  \label{eqn:multivariate_series1}
\end{equation} where $\z^{\wp(w)} = z_1^{|w|_{a_1}}\cdots z_d^{|w|_{a_d}}$. The series \eqref{eqn:multivariate_series1} is a multivariate series associated with $L$ and can be rewritten as 

 \begin{equation}
    \Gamma_L(\z) =  \displaystyle \sum_{\ii \in \N^d} \gamma_{\ii}\z^{\ii},  \label{eqn:multivariate_series2}
\end{equation} where $\gamma_{\ii}$ is the number of words $w$ in $L$ with $\wp(w) = \ii.$ \\

If we normalize the vector $\wp(w)$ as $\displaystyle\frac{1}{|w|}\wp(w)$ we get a vector of frequencies $\Tilde{\wp}(w)$ belonging to the simplex $M_d$ of probability vectors: 
$$M_d = \left\{\rr\in \R_{\geq 0}^d : \rr = (r_1,\cdots,r_d), r_i \geq 0, \displaystyle \norm{\rr} = \sum_{i=1}^d r_i = 1\right\}$$ and $\norm{\cdot}$ is a $l_1$ norm. The multivariate growth \emph{indicatrice} that we are going to introduce is the number $\psi(\rr),$ $\rr \in M_d$ which characterizes the growth of coefficients $\gamma_{\ii}$ when $\norm{\ii} \rightarrow \infty$ in the direction of vector $\rr.$ When $\rr \in \Q^d$ is a rational vector, then the possible approach would be to define $\psi(\rr)$ by 
\begin{equation}
    \psi(\rr) =\displaystyle \limsup_{n \rightarrow \infty} \frac{1}{n} \log|\gamma_{n\rr}| \label{eqn:multigrow_rational}
\end{equation}
 where $n \in \N$ and $\gamma_{n\rr} = 1$ if $n\rr \notin \N^d.$ The Definition \ref{def:psi_def} given in the Section \ref{sec:cond_Q} follows the idea of \cite{quint2002divergence}. It works for arbitrary $\rr \in M_d$ and coincides with \eqref{eqn:multigrow_rational} in the rational case for many examples.  \\
 
 A crucial assumption that we make is the ``concavity'' assumption (CG) (see Definition \ref{def:condition_Q}) which allows to apply the powerful methods from convex analysis (as well as the results from \cite{quint2002divergence}). In fact our definition works for arbitrary multivariate power series \eqref{eqn:multivariate_series2} with real coefficients. The main point is to present the indicatrice of growth $\psi(\rr)$ (assuming the condition (CG)) as
 \begin{equation}
     \psi(\rr) = \displaystyle \inf_{\theta \in \partial(\Omega')} \langle \rr, \theta \rangle \label{eqn:psi_inf1}
 \end{equation}
 where $\Omega' = - \Omega \subset \R^d_{\geq 0}$ and $\Omega$ is a closed convex set representing the $\relog$ image of the domain of absolute convergence of \eqref{eqn:multivariate_series2} where
 $$ \relog(\z) = \left(\log|z_1|,\cdots,\log|z_d|\right)$$ (see Theorem \ref{the:psi_as_support}). We apply \eqref{eqn:psi_inf1} for two languages: the Fibonacci language $L_{Fib}$ and the language $L_{F_m}$ of freely reduced words associated with a free group $F_m$ of rank $m\geq 2$. These languages belong to class of regular languages, that is languages accepted by finite automaton. Regular languages play important role in many areas of mathematics, including dynamical systems and algebra. Regular normal form of elements in the group is a bijective presentation of elements of the group by elements of a regular language over the alphabet of generators and inverses. Regular geodesic normal form is such presentation for which the length of the element with respect to generating set coincides with the length of the word. Virtually abelian groups and Gromov hyperbolic groups have a regular geodesic normal form for any system of generators.\\
 
 Regular languages are good in particular because their growth series (in one or multivariate case) are rational functions. This fact even in stronger form was known already to Chomsky and Sch$\ddot{u}$tzenberger \cite{chomsky1959algebraic}. Proposition \ref{pro:multigrow} of Section \ref{sec:cond_E} gives a rational expression for a multivariate growth series associated with regular language. The condition (CG) mentioned earlier holds under the assumption of ergodicity of the automaton presenting the language (condition (E) in Section \ref{sec:cond_E}). It is satisfied in the presented examples and we summaries the computations from Section \ref{sec:indicator_F_m} and the Propositions \ref{pro:fib_psi} as
 {\The The indicatrices $\psi_{F_2}(\rr)$ and $\psi_{Fib}(\rr)$ 
 are given by
 \begin{enumerate}
     \item $$\psi_{F_2}(\rr) = \mathbf{H}(\rr) + p \displaystyle\log\left(2q-p+2\sqrt{p^2-pq+q^2}\right) $$ 
     $$ + q\displaystyle\log\left(2p-q+2\sqrt{p^2-pq+q^2}\right),$$
     \item 
     \begin{eqnarray}
      \psi_{Fib}(\rr) &=& p \log \displaystyle \left(\frac{p}{p-q}\right) + q\log \displaystyle\left( \frac{p-q}{q}\right)~~~ \textnormal{ if } p \geq \displaystyle\frac{1}{2}, \nonumber\\
      &=& -\infty~~~ \textnormal{ if } p < \displaystyle\frac{1}{2}, \nonumber
\end{eqnarray}
 \end{enumerate}
 where  $ \rr = (p,q) \in M_2$} and $\mathbf{H}(\rr) = -p\log p -q\log q$ is the Shannon's entropy.\\
 
 In fact, $\psi_{F_2}(\rr)$ is computed for modified multivariate growth series $\Delta_{F_2}(\z),$ $ \rr \in M_2$ (see Section \ref{sec:Fm_gr}). The graphics of these functions are presented by Figure \ref{fig:indicatrice}.
 
 \begin{figure}[!htb]
\begin{subfigure}[b]{0.45\textwidth}
\centering
    \includegraphics[width=\textwidth]{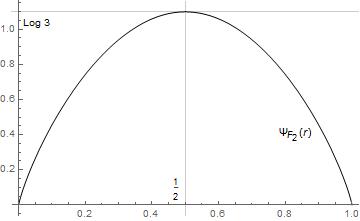}
    \caption{}
   \label{subfig:ind_F2}
\end{subfigure}
\hfill
\begin{subfigure}[b]{0.45\textwidth}
\centering
    \includegraphics[width=\textwidth]{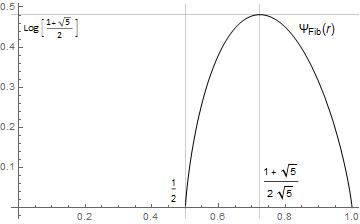}; 
    \caption{}
    \label{subfig:ind_Fib}
\end{subfigure}
\caption{Graphs of $\psi_{F_2}(\rr)$ and $\psi_{Fib}(\rr)$.}
\label{fig:indicatrice}
\end{figure}
 
 Regular languages quite often appear in Dynamical systems, first of all as languages associated with subshifts of finite type (SFT). For example a subshift of a full shift $\Sigma^{\Z}$ determined by a finite set ${\bf F} $ of forbidden words; $\Sigma = \{0,1\}$, ${\bf F} = \{11\}$, is the Fibonacci language. And if $\Sigma = \{a_1,\cdots,a_m,a_1^{-1},\cdots,a_m^{-1}\}$, ${\bf F}  = \{a_ia_i^{-1},a_i^{-1}a_i: i=1,\cdots,m\}$ we get a subshift corresponds to the language $L_{F_m}$ of freely reduced words representing  elements of the free group $F_m$ over $\Sigma$.\\
 
 A standard way to present SFT is to define it by a digraph $\Gamma = (V,E)$ (i.e. a directed graph with allowed loops and multiple edges) or equivalently by a matrix $A$ with non-negative integer entries. The theory of positive matrices based first of all on the powerful Perron-Frobenius theorem allows comprehensive study of SFT from the dynamical and other points of view. There is a canonical way to associate with irreducible SFT an ergodic Markov shift (and a Markov chain on $\Sigma$) given by a stochastic matrix $P.$ Let $(\Delta,\zeta)$ be a SFT where $\Delta \subset \Sigma^{\Z}$ and $\zeta:\Sigma^{\Z} \rightarrow \Sigma^{\Z}$ is a shift map defined by $(\zeta x)_n = x_{n+1}, x \in \Sigma^{\Z}.$ Let $L(\Delta)$ be the language of $(\Delta,\zeta)$ (all finite words that appears as a subwords in $x \in \Sigma^{\Z}$). Let $\psi_{\Delta}(\rr)$ be indicatrice of growth of $L(\Delta)$. The maximum of values of $\psi_{\Delta}(\rr)$ is a \emph{topological entropy} of $L(\Delta)$ \cite{lind1995introduction} and $\psi_{\Delta}(\rr)$ satisfies the relation \eqref{eqn:psi_inf1}, because $L(\Delta)$ is a regular language and the condition (CG) holds. There is an alternative way to present $\psi_{\Delta}(\rr)$ as the support function of a convex set via the methods of Large Deviation Theory (LDT) \cite{zeitouni1998large}. For our purpose it is enough only to apply Sanov's theorem stating that the Large Deviation Principle holds for finite Markov chains. The discussion in Section \ref{sec:LDT} relates formula \eqref{eqn:psi_inf1} with the Sanov's expression for the rate function $I(\rr).$ \\

 The paper is organized as follows. In Section \ref{sec:cond_E}, we recall some of the basic definitions from the theory of finite automata and formal languages that will be needed later. Then we introduce the condition (E) for the regular languages. The Section \ref{sec:Fm_gr} is devoted to the computations of the modified multivariate growth series of language $L_{F_m}$ of reduced elements of a free group $F_m$ using two approaches. In Section \ref{sec:cond_Q}, we discuss the condition (CG) and then prove Theorem \ref{the:psi_as_support}. In Sections \ref{sec:indicator_F_m} and \ref{sec:indicator_fib} we present computations of indicatrice $\psi(\rr)$ for $F_2$ and for the Fibonacci language $L_{Fib}$, respectively. Section \ref{sec:LDT} is devoted to application of Large Deviations Theory. Using Sanov Theorem, we get a relationship between $\psi(\rr)$ and the rate function $I(\rr).$ Using result from asymptotic combinatorics in several variables presented in \cite{melczer2020invitation}, in Section \ref{sec:alter}, we get a finer asymptotics associated with $F_2$. Finally, Section \ref{sec:open} contains concluding remarks and some open questions. 
 
\section{Regular languages, their growth series and the condition (E)}\label{sec:cond_E}

We begin this section with recall of definition of finite automaton and language accepted by it. A finite automaton $\A$ is given by a quintuple $\left(Q,\Sigma,\kappa,q_0,\F\right),$ where $Q$ is a finite set whose elements are called states, $\Sigma$ is a finite alphabet, $\kappa:Q\times\Sigma\rightarrow Q$ is a transition function, the state $q_0 \in Q$ is a special state called the initial state and the set $\F \subset Q$ is nonempty set whose elements are called final states. It is convenient to visualize $\A$ as a labeled directed graph $\Theta_{\A}$ with the vertex set $Q,$ edge set 
$$E = \{(q,s) : q,s \in Q, \kappa(q,a_i) = s \textrm{ for some } a_i \in \Sigma\},$$
and each such edge $(q,s)$ with $\kappa(q,a_i) = s$ is supplied by the label $a_i$. Multiple edges and loops are allowed. The graph $\Theta_{\A}$ is called the diagram of $\A.$ The example of these diagrams are presented by Figures \ref{subfig:unambiguous} and \ref{fig:fibonacci}. Observe that so defined $\A$ is deterministic and complete automaton, i.e. given any $q\in Q$ and any $a \in \Sigma$ we know what would be the next state $\kappa(q,a).$ A word $w \in \Sigma^*$ is accepted by $\A$ if starting with the initial state $q_0$ and traveling in diagram $\Theta_{\A}$ along with the path $p_w$ determined by $w$ we end up at some final state. Let $\La\left(\A\right)$ be the set of words accepted by $\A.$ A language $L \subset \Sigma^*$ is called regular if there is a finite automaton $\A$ such that $L = \La\left(\A\right).$ The important feature of this definition is the uniqueness of the path $p_w$ for each $w \in \Sigma^*.$ \\

One can generalize the above definition by replacing a singleton $\{q_0\}$ by a nonempty subset $\I \subset Q$ whose elements are called initial states and defining $\La\left(\A\right)$ as a set of words $w$ for which there is an initial state $i \in \I$ such that the path $p_{i,w}$ that begins at $i$ and follow the word $w$ ends up at $\F.$ Surprisingly, this does not lead to a larger class of languages (as it is always possible to replace $\A$ by automaton $\A'$ with a single initial state such that $\La\left(\A\right)=\La\left(\A'\right)$). The automaton with a single initial state are unambiguous in the sense that for each $w \in \La\left(\A\right)$ there a unique path $p_w$ that recognizes $w.$ Nevertheless, there are situations (for instance in the case of the language of freely reduced words over the alphabet of generators of a free group) when ambiguous automata work better (see Figure \ref{subfig:ergodic} and Section \ref{sec:indicator_F_m}). \\

One also can consider non-deterministic and incomplete automata. Still, this much larger class of automata determines the same class of languages, the class of regular languages. Non-deterministic automata appear for instance in the study of languages associated with sofic subshifts \cite{lind1995introduction}. 

\begin{figure}[!htb]
\begin{subfigure}[b]{0.2\textwidth}
\centering
    \begin{tikzpicture}
        \node[state, accepting, initial] (q0) {$s_0$};
        \node[state, accepting, above of=q0] (qa) {$s_1$};
        \node[state, accepting, below of=q0] (qa-) {$s_3$};
        \node[state, accepting, left of=q0] (qb) {$s_2$};
        \node[state, accepting, right of=q0] (qb-) {$s_4$};
       
         \draw[->]  (q0) edge[above] node[right]{$a$} (qa)
               (q0) edge[below] node[left]{$a^{-1}$} (qa-)
               (q0) edge[bend left] node[above]{$b$} (qb)
               (q0) edge[bend right] node[above]{$b^{-1}$} (qb-)
               
               (qa) edge[loop above] node[above]{$a$} (qa)
               (qa) edge[bend right] node[above]{$b$} (qb)
               (qa) edge[bend left] node[above right]{$b^{-1}$} (qb-)

               (qa-) edge[loop below] node[below]{$a^{-1}$} (qa-)
               (qa-) edge[bend left] node[below]{$b$} (qb)
               (qa-) edge[bend right] node[below]{$b^{-1}$} (qb-)
               
               (qb) edge[loop left] node[below]{$b$} (qb)
               (qb) edge[above right] node[right]{$a$} (qa)
               (qb) edge[below right] node[above right]{$a^{-1}$}
               (qa-)
               
               (qb-) edge[loop right] node[above=2pt]{$b^{-1}$} (qb-)
               (qb-) edge[above left] node[left]{$a$} (qa)
               (qb-) edge[below left] node[left]{$a^{-1}$}
               (qa-);
    \end{tikzpicture}
    \caption{}
   \label{subfig:unambiguous}
\end{subfigure}
\hfill
\begin{subfigure}[b]{0.35\textwidth}
\centering
    \begin{tikzpicture}
        \node[state, initial, accepting] (q1) {$s_1$};
        \node[state, initial, accepting, right of=q1] (q3) {$s_2$};
        \node[state, initial, accepting, above of=q1] (q2) {$s_4$};
        \node[state, initial, accepting, above of=q3] (q4) {$s_3$};
        \draw  [->] (q1) edge[bend left, left] node{$b^{-1}$} (q2)
                (q1) edge[loop below] node{$a$} (q1)
                (q2) edge[loop above] node{$b^{-1}$} (q2)
                (q4) edge[loop above] node{$a^{-1}$} (q4)
                (q2) edge[bend left, left] node{$a$} (q1)
                (q2) edge[bend left, above] node{$a^{-1}$} (q4)
                (q3) edge[loop below] node{$b$} (q3)
                (q3) edge[bend left, above] node{$a$} (q1)
                (q1) edge[bend left, below] node{$b$} (q3)
                (q4) edge[bend left, above] node{$b^{-1}$} (q2)
                (q4) edge[bend left, right] node{$b$} (q3)
                (q3) edge[bend left, left] node{$a^{-1}$} (q4);
    \end{tikzpicture}
    \caption{}
    \label{subfig:ergodic}
\end{subfigure}
\caption{Diagrams of automata $\A$ and $\A'$ associated with the language $L_{F_2}$ of reduced words representing elements of the free group $F_2$, where the initial states are denoted by horizontal unlabeled arrow and the final states are represented by ``double'' circles.}
\label{fig:ergodic_and_nonergodic}
\end{figure}
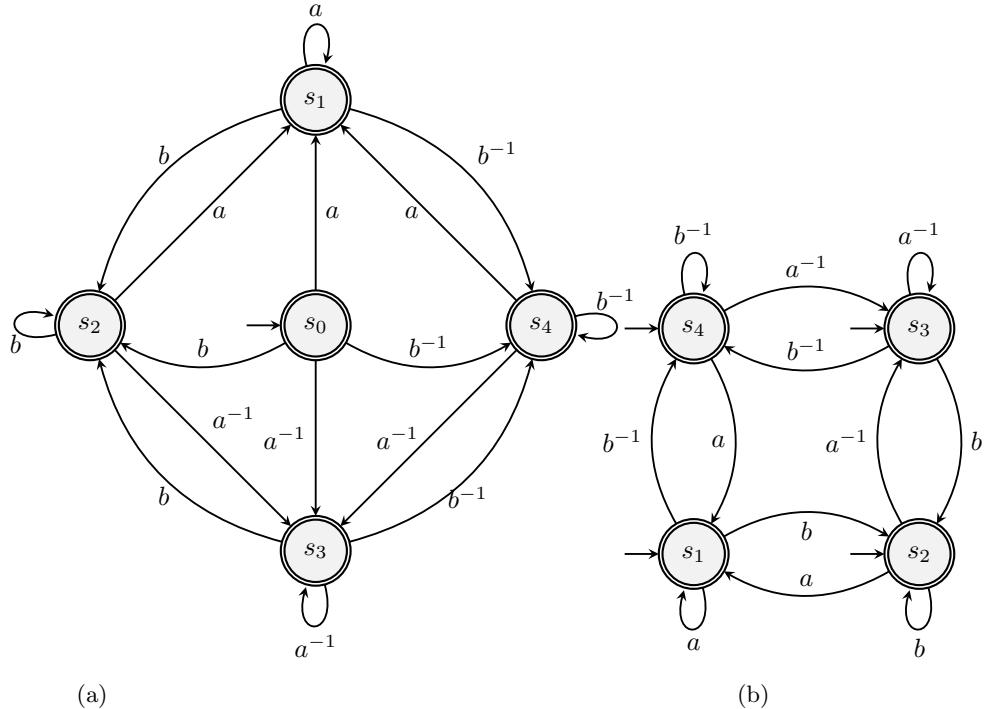

Recall that given $w \in \Sigma^*$ the vector $\wp(w) = \left(|w|_{a_1},\cdots,|w|_{a_d}\right),$ where $|w|_{a_i}$ denote the number of occurrences of $a_i \in \Sigma$ in $w.$ Let also $|w|$ denote the length of $w.$ With $L \subset \Sigma^*$ one can associate a number of formal series: the ordinary series, 
\begin{equation}
    \Gamma_L(z) = \displaystyle \sum_{w \in L} z^{|w|}, ~~z \in \Co \label{eqn:ordinary_series}
\end{equation}
the multivariate series,
\begin{equation}
    \Gamma_L(\z) = \Gamma_L(z_1,\cdots,z_d) = \displaystyle \sum_{w \in L} \z^{\wp(w)},  \label{eqn:multivariate_series}
\end{equation} where $\z = (z_1,\cdots,z_d) \in \Co^d$ and $\z^{\wp(w)} = z_1^{|w|_{a_1}}\cdots z_d^{|w|_{a_d}}$. Also one can consider the formal sum
\begin{equation}
   \displaystyle \sum_{w \in L} wz^{|w|} \in \Z[\Sigma^*]\llbracket z \rrbracket \label{eqn:formal_sum}
\end{equation}
viewed as a formal power series with coefficients in the ring $\Z[w^*]$ (the semi-group ring of the free semi-group $\Sigma^*$). The consideration of these type of series go back at least to 50's of 20th century and is related first of all with the names of Chomsky and Sch$\ddot{u}$tzenberger \cite{chomsky1959algebraic}.\\

For us it is important that in the case when $L$ is a regular language the series \eqref{eqn:ordinary_series}, \eqref{eqn:multivariate_series} and \eqref{eqn:formal_sum} are rational i.e. ratio of two polynomials. We focus our attention only on \eqref{eqn:ordinary_series} and \eqref{eqn:multivariate_series}, and in fact the study of asymptotic properties of \eqref{eqn:multivariate_series} is the main purpose of this article.\\

Let $A = (a_{ij})$ be the adjacency matrix of $\Theta_{\A}$ i.e. a $|Q|\times|Q|$ matrix whose rows and columns correspond to the states and 
$$a_{ij} = \textrm{ the number of edges in }\Theta_{\A} \textrm{ joining } i \textrm{ with } j, \textrm{ where } i,j \in Q.$$
We use the ordering on $Q$ in such a way that the first state is $q_0$. Let  $$u = \left(1,0,\cdots,0\right), \textrm{ and } v = \left(v_1,\cdots,v_{|Q|}\right)^t$$ be a row and column vectors of dimension $|Q|$, where 
$$v_q = 1, \textrm{ if  state } q \in \F \textrm{ and } v_q = 0 \textrm{ if } q \notin \F.$$
Then the standard technique of counting paths (see Theorem 4.7.2 of \cite{stanley2011enumerative}) in finite graph (or in finite Markov chain) leads to the 
$$ \Gamma_L(z) = u\displaystyle\left[\sum_{n\geq 0} (zA)^n\right]v = u \left[I-zA\right]^{-1}v$$
and the rationality of $L(z)$ is obvious. A similar argument works for multivariate case, only the matrix $zA$ should be replaced by 
$$A(\z) = \left(a_{st}(\z)\right)_{s,t=1}^{|Q|},$$ where $\z= (z_1,\cdots,z_d)$ and
 \[  a_{st}(\z) = \left\{
\begin{array}{ll}
      \displaystyle \sum_{i}z_i & \textrm{ where summation is taken over such } i \textrm{ that } \kappa(s,a_i)=t \\
      0 & \textrm{ if there is no edge from }s \textrm{ to }t. \\
\end{array} 
\right. \]
With this notations we have the following.

\begin{Pro}\label{pro:multigrow}
The multivariate growth series of a regular language $L = \La\left(\A\right)$ satisfies
\begin{equation}
    \Gamma_L(\z) = u\left[I-A(\z)\right]^{-1}v. \label{eqn:multigrow}
\end{equation}
\end{Pro}

Hence $\Gamma_L(\z), \z \in \Co^d$ is a multivariate rational function i.e. 
$$ \Gamma_L(\z) = \displaystyle\frac{G(\z)}{H(\z)},$$ where $G(\z), H(\z) \in \Co[z_1,\cdots,z_d].$
The polynomials $G(\z), H(\z)$ obtained in \eqref{eqn:multigrow}, can be calculated, for instance, using the Cramer's rule.\\

For example, in the case of automaton presented by Figure \ref{fig:fibonacci},
$$A(\z) =  \left(\begin{array}{cc}
 z_1 & z_2 \\
 z_1 & 0 \\
\end{array}\right) \textrm{ and } \Gamma_L(\z) = \displaystyle \frac{1+z_2}{1-z_1-z_1 z_2}.$$

For us the following condition will be crucial. 

\begin{Def}\label{def:condition_E} (Condition (E)). 
We say that a language $L \subset \Sigma^*$ satisfy the condition (E) if there exists an integer $N \geq 0$ such that, for every $U, V \in L,$ there exists $w \in \Sigma^*$, $|w| \leq N$ with $UwV \in L.$
\end{Def}

The example of the regular languages satisfying the condition (E) come from ergodic (or primitive) automata $\A$ i.e. automata in which any state can be connected to another state by a path in the diagram. 

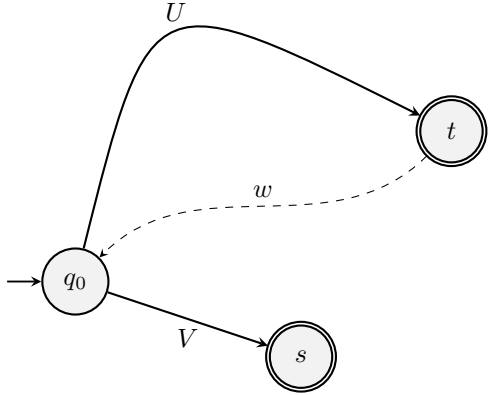
\begin{figure}[!htb]
    \centering
    \begin{tikzpicture}[scale=1]
  
  \node[state, initial, draw] at (0,0) (a) {$q_0$};
  \node[state, accepting] at (3,-1) (b) {$s$};
  \node[state, accepting, draw] at (5,2) (c) {$t$};
  
  \draw[thick,->] (a) -- (b) node[midway, below,black]{$V$};
  \draw[dashed,->] (c)  to[out=225, in=45] node[midway, above]{$w$} (a) ;
  \draw[thick,->] (a).. controls (1,4) .. (c) node[midway, above]{$U$};
\end{tikzpicture}
    \caption{Part of Moore diagram of an ergodic automaton $\A$}
    \label{fig:strongly_connectedness}
\end{figure}

This is because $\A$ is ergodic automaton,  $U, V \in \La\left(\A\right)$ and $s,t$ are end states of the path $p_U, p_V$ then connecting state $s$ with the initial state $q_0$ by some path $q$ whose length does not exceed the diameter $D$ of the graph $\Theta_{\A}$ (i.e. maximum of combinatorial distances between states in $\A$) we get $UwV \in  \La\left(\A\right)$ where $w$ is a word read along path $q,$ $|w| \leq D.$ See Figure \ref{fig:strongly_connectedness}. So the condition (E) holds.\\

Unfortunately, not every regular language satisfying (E) can be accepted by ergodic automaton of the type described above. For instance the language $L_{F_2}$ over the alphabet $\Sigma = \{a,a^{-1},b,b^{-1}\}$ of freely reduced words (i.e. $aa^{-1},a^{-1}a,$ \\ $bb^{-1},b^{-1}b$ are forbidden) of the free group $F_2$ of rank $2$, satisfy the condition (E) but can not be accepted by an ergodic unambiguous automaton \cite{ Tulio2003entropy}. The Figure \ref{fig:ergodic_and_nonergodic} show two automata accepting the language $L_{F_2}.$ The first automaton presented in the Figure \ref{subfig:unambiguous} is unambiguous but not ergodic. In the second automaton \ref{subfig:ergodic}, all states are initial, it is ergodic automaton but not unambiguous. \\

Later, we will introduce a condition (CG) for multivariate power series and it will follow that $\Gamma_L(\z), \z \in \Co^d$ satisfies the condition (CG) if the associated language $L$ satisfies the condition (E). 

\section{Computations of the multivariate growth series in the case of a free group}\label{sec:Fm_gr}
Let $F_m = \langle a_1,\cdots,a_m\rangle$ be a free group of rank $m \geq 2$ and $L_{F_m}$ be the language of freely reduced words over alphabet $\Sigma = \{a_1,\cdots,a_m,a_1^{-1},\cdots,a_m^{-1}\}.$ This is a regular language (an automaton accepting it in the case $m=2$ is given by Figure \ref{subfig:unambiguous}) and hence a corresponding multivariate growth series $\Gamma_{F_m}(\z)$ is rational. It can be computed using for instance Proposition \ref{pro:multigrow}. Even in the case of $m=2$ computations are quite involved. So instead we will focus on a modified version $\Delta_{F_m}(\z)$ defined as 
$$\Delta_{F_m}(\z) = 1 + \displaystyle \sum_{g \in F_m,~g \neq e} \z^{\wp(g)}$$
where $\wp(g) = (|g|_{a_1^{\e}},\cdots,|g|_{a_m^{\e}})$ and $|g|_{a_i^{\e}}$ is a number of occurrences of symbols $a_i$ and $a_i^{-1}$ in the freely reduced word presenting element $g$ i.e. $|g|_{a_i^{\e}} = |g|_{a_i} + |g|_{a_i^{-1}}.$ The modified series $\Delta_{F_m}(\z)$ is also important for us, because of the question \ref{que:chi_varphi} formulated at the end of the article in connection with symmetric random walks. Of course if we know $\Gamma_{F_m}(\z)$, then we can obtain the $\Delta_{F_m}(\z)$ and vice versa.\\

We present the computations of the modified multivariate growth series $\Delta_{F_m}(\z)$, first using the automata approach i.e using the Proposition \ref{pro:multigrow} and later, using the symmetries of this language leading to a simpler system of equations and as a result of simpler expression for $\Delta_{F_m}(\z)$ given by Proposition \ref{pro:delta_Fm}.

\begin{Pro}
\label{pro:Fm_growth}
\begin{equation}
    \Gamma_{F_m}(\z) = \displaystyle\frac{(1+z_1)\cdots (1+z_m)}{R(\z)}  \label{eqn:Fm_2.1}
\end{equation}
where 
   $$ R(\z) = 1-\displaystyle  \left( \sum_{i} z_i + 3\displaystyle\sum_{i < j} z_iz_j + \cdots + (2l-1)\displaystyle
    \sum_{i_1< \cdots < i_l}  \prod_{k=1}^l z_{i_k} \right.  $$
$$\left. + \cdots +(2m-3)\displaystyle\sum_{i_1 < \cdots < i_{m-1}} \prod_{k=1}^{m-1} z_{i_k} +(2m-1)z_1\cdots z_m \right).
$$
\end{Pro}
\noindent
\begin{proof}
Let $\A_{F_m}$ be an automaton generating language $L_{F_m}$ of reduced elements of $F_m$. See Figure \ref{subfig:unambiguous} for $\A_{F_2}$. 
The $2m+1$ states of $\A_{F_m}$ are $s_0,s_1,\cdots,s_m,s_{m+1}= s_1^{-1},\cdots,s_{2m} = s_m^{-1}$ and the adjacency matrix $A$ of $\A_{F_m}$ is
$$A = \left(
\begin{array}{ccc}
  \oo_{1 \times 1} & J_{1\times m} & J_{1\times m} \\
 \oo_{m\times 1} & J_{m\times m} & (J-I)_{m\times m} \\
 \oo_{m\times 1} &  (J-I)_{m\times m} & J_{m\times m}
\end{array}\right),$$
where $\oo_{m\times n}$ is zero matrix, $J_{m\times n}$ is matrix with all entries are one and \\$(J-I)_{m\times m}$ is matrix having diagonal entries are zero and one elsewhere. Using the above decomposition of $A$, we obtain the transfer matrix $A(\z)$ and then applying Equation \eqref{eqn:multigrow} of the Proposition \ref{pro:multigrow} we obtain
\begin{equation}
    \label{eqn:rational_expn_d_var_growth}
   \Gamma_{F_m}(\z) = \displaystyle \frac{\sum_{i} (-1)^{i+1} \det(I-A(\z):i,1)}{\det\left(I-A(\z)\right)},
\end{equation}
where sum is taken over all final states $s_i$ and if $N$ is matrix then $(N : i,j)$ denotes the matrix obtained by removing the $i$th row and $j$th column of $N$. The matrix $I$ is the $2m+1$ dimensional identity matrix.\\

We divide the rest of the proof in two lemmas. 

\begin{lem} \label{lemma:det(I-Az)}
$$\det\left(I-A(\z)\right) = (1-z_1)\cdots (1-z_m) \times R(\z)$$
\end{lem}
\begin{proof}
We prove the lemma using induction on $m \geq 2$. It is easy to see that the statement is true when $m=2$. i.e. 
$$\det\left(I-A(\z)\right) = (1-z_1)(1-z_2)(1-z_1-z_2-3z_1z_2).$$ 
Assume that the statement is true when $m=k$. i.e.
$$\det\left(I-A(\z)\right) = (1-z_1)\cdots (1-z_k) \times$$
$$ \left(1-\displaystyle  \left( \sum_{i} z_i + 3\displaystyle\sum_{i < j} z_iz_j + \cdots + \left(2l-1)\right)\displaystyle
    \sum_{i_1< \cdots < i_l} \prod_{j=1}^l z_{i_j} \right.\right.  $$
$$\left.\left. + \cdots +(2k-3)\displaystyle\sum_{i_1< \cdots < i_{k-1}} \prod_{j=1}^{k-1} z_{i_j} +(2k-1)z_1\cdots z_k \right)\right).
$$
Let $m=k+1$. Then $I-A(\z)$
$$ = \left(
\begin{array}{c|ccc|ccc}
1 & -z_1 & \cdots & -z_{k+1} & -z_1 & \cdots & -z_{k+1} \\
\hline
0 &1 -z_1 & \cdots & -z_{k+1} & 0 & \cdots & -z_{k+1} \\ 
\vdots &\vdots & \ddots & \vdots & \vdots & \ddots & \vdots \\
0 & -z_1 & \cdots & 1-z_{k+1} & -z_1 & \cdots & 0 \\
\hline
0 & 0 & \cdots & -z_{k+1} & 1-z_1 & \cdots & -z_{k+1} \\ 
\vdots &\vdots & \ddots & \vdots & \vdots & \ddots & \vdots \\
0 & -z_1 & \cdots & 0 &- z_1 & \cdots & 1-z_{k+1} \\
\end{array}\right)$$
We rearrange the rows and columns of $I-A(\z)$ by interchanging the positions of $(k+1)$th row and column with $\{k+2,\cdots,2k+1\}$th rows and columns, respectively. 
Then $I-A(\z)$
$$= \left(
\begin{array}{c|cccccc|cc}
1 & -z_1 & \cdots & -z_k & -z_1 & \cdots & -z_k & -z_{k+1}& -z_{k+1} \\
\hline
0 & 1 -z_1 & \cdots & -z_k & 0 & \cdots & -z_k & -z_{k+1} & -z_{k+1} \\ 
\vdots &\vdots & \ddots & \vdots & \vdots & \ddots & \vdots & \vdots & \vdots \\
0 & -z_1 & \cdots & 1-z_k & -z_1 & \cdots & 0 & -z_{k+1} & -z_{k+1} \\
0 & 0 & \cdots & -z_k & 1-z_1 & \cdots & -z_k & -z_{k+1} & -z_{k+1} \\ 
\vdots &\vdots & \ddots & \vdots & \vdots & \ddots & \vdots & \vdots & \vdots \\
0 & -z_1 & \cdots & 0 & -z_1 & \cdots & 1-z_k & -z_{k+1} & -z_{k+1} \\
\hline
0 & -z_1 & \cdots & -z_k & -z_1 & \cdots & -z_k & 1-z_{k+1} & 0 \\
0 & -z_1 & \cdots & -z_k & -z_1 & \cdots & -z_k & 0 & 1-z_{k+1} \\
\end{array}\right)$$
Observe that the rows and columns of $I-A(\z)$ are indexed as $s_0, s_1,\cdots,s_k,s_1^{-1},$\\$\cdots,s_k^{-1},s_{k+1},s_{k+1}^{-1}$ and the determinant of the central block and the expression that we have assumed when $m=k$ are identical. As first column has exactly one 1 and zero elsewhere, we can ignore the first row and first column. Hence $\det\left(I-A(\z)\right)$
$$ = \det\left(
\begin{array}{cccccc|cc}
 1 -z_1 & \cdots & -z_k & 0 & \cdots & -z_k & -z_{k+1} & -z_{k+1} \\ 
\vdots & \ddots & \vdots & \vdots & \ddots & \vdots & \vdots & \vdots \\
 -z_1 & \cdots & 1-z_k & -z_1 & \cdots & 0 & -z_{k+1} & -z_{k+1} \\
 0 & \cdots & -z_k & 1-z_1 & \cdots & -z_k & -z_{k+1} & -z_{k+1} \\ 
\vdots & \ddots & \vdots & \vdots & \ddots & \vdots & \vdots & \vdots \\
 -z_1 & \cdots & 0 & -z_1 & \cdots & 1-z_k & -z_{k+1} & -z_{k+1} \\
\hline
-z_1 & \cdots & -z_k & -z_1 & \cdots & -z_k & 1-z_{k+1} & 0 \\
-z_1 & \cdots & -z_k & -z_1 & \cdots & -z_k & 0 & 1-z_{k+1} \\
\end{array}\right) $$
We rewrite the matrix by naming these blocks as
$$\det\left(I-A(\z)\right)= \det\left( \begin{array}{cc}
  A_{2k\times 2k}  & B_{2k\times 2} \\
  C_{2\times 2k}  & D_{2\times 2}
\end{array}\right)$$
Recall that $z_i \neq 1$, for all $i = 1,\cdots,k+1$ and hence $\det(D_{2\times 2}) \neq 0$. Therefore we have
$$\det\left(I-A(\z)\right)= \det\left( \begin{array}{cc}
  A'_{2k\times 2k}  & B_{2k\times 2} \\
   \oo_{2\times 2k} & D_{2\times 2}
\end{array}\right) = \det( A'_{2k\times 2k})\det(D_{2\times 2}),$$
where $A' = A_{2k\times 2k}-B_{k\times 2}D_{2\times 2}^{-1}C_{2\times k}$ and its $(i,j)$th entry 
$$ A'(i,j) = \left\{
\begin{array}{ll}
      1-z_i-\displaystyle\frac{2z_iz_{k+1}}{1-z_{k+1}} & \textrm{ if  } i = j \nonumber\\
      -\displaystyle\frac{2z_iz_{k+1}}{1-z_{k+1}} & \textrm{ if  } i < j \textrm{ and  } j = i+k(\bmod{~2k})\nonumber\\
      -\displaystyle\frac{2z_iz_{k+1}}{1-z_{k+1}} & \textrm{ if  } j < i \textrm{ and  } i = j+k(\bmod{~2k})\nonumber\\
      -z_i-\displaystyle\frac{2z_iz_{k+1}}{1-z_{k+1}} & \textrm{ otherwise.} \nonumber
\end{array} 
\right.$$
In order to prove the $m=k+1$ step, it is suffices to show that 
$$\det(A') = \displaystyle\frac{(1-z_1)\cdots (1-z_k)}{1-z_{k+1}}\times$$
$$\left(1-\sum_{i} z_i - 3\sum_{i< j} z_iz_j - \cdots -(2k-1)\sum_{i_1< \cdots < i_k} \prod_{j=1}^k z_{i_j}-(2k+1)z_1\cdots z_{k+1}\right).$$
Applying below rows and columns transformations 
\begin{enumerate}[noitemsep]
    \item $R_i-R_1, i=2,\cdots,2k$
    \item $C_i-C_{k+i}, i = 1,\cdots,k$
    \item $R_i+R_1, i = 2,\cdots,k,k+2,\cdots,2k$ and $R_{k+1}+2R_1$
    \item $R_{k+i}+R_i, i = 2,\cdots,k$
    \item $R_i+R_1, i = 2,\cdots,k$ and $R_i+2R_1, i = k+1,\cdots,2k$
\end{enumerate}
 to the matrix $A'$, we convert $A'$ to the below upper triangular block matrix. $\det(A')$
$$ = \det\left( \begin{array}{ccc|ccc}
 1-z_1 & \cdots & 0 & -\frac{2z_1z_{k+1}}{1-z_{k+1}} & \cdots & -z_1-\frac{2z_kz_{k+1}}{1-z_{k+1}}\\
\vdots & \ddots & \vdots & \vdots & \ddots & \vdots \\
0 & \cdots & 1-z_k & -z_k-\frac{2z_1z_{k+1}}{1-z_{k+1}} & \cdots & -\frac{2z_kz_{k+1}}{1-z_{k+1}}\\
\hline
0 & \cdots & 0 & 1-z_1-\frac{4z_1z_{k+1}}{1-z_{k+1}} & \cdots & 2\left(-z_k-\frac{2z_kz_{k+1}}{1-z_{k+1}}\right)\\
\vdots & \ddots & \vdots & \vdots & \ddots & \vdots \\
0 & \cdots & 0 & 2\left(-z_1-\frac{2z_1z_{k+1}}{1-z_{k+1}}\right) & \cdots & 1-z_k-\frac{4z_kz_{k+1}}{1-z_{k+1}}\\
\end{array}\right)$$
We call the bottom right corner block as $A''$. In $A''$ matrix, subtracting twice the first row from each row $R_i, i = 2,\cdots, k$, we get
$$\det(A'') = \det\left(\begin{array}{cccc}
 \frac{1-z_1-z_{k+1}-3z_1z_{k+1}}{1-z_{k+1}} & \frac{-2z_2(1+z_{k+1})}{1-z_{k+1}} & \cdots & \frac{-2z_k(1+z_{k+1})}{1-z_{k+1}}\\
-1-z_1 & 1+z_2 &\cdots & 0 \\
\vdots & \vdots & \ddots & \vdots \\
-1-z_1 & 0 & \cdots & 1+z_k \\
\end{array}\right)$$
$$ = \left(1-\sum_{i} z_i - 3\sum_{i< j} z_iz_j - \cdots -(2k-1)\sum_{i_1< \cdots < i_k} \prod_{j=1}^k z_{i_j} -(2k+1)z_1\cdots z_{k+1}\right).$$
This implies that the statement is true for $m=k+1$ and hence the lemma is proved.
\end{proof}
\begin{lem} \label{lemma:minor_sum}
$$\sum_{i} (-1)^{i+1} \det(I-A(\z):i,1)= (1-z_1^2)\cdots(1-z_m^2).$$
\end{lem}
\begin{proof}
Using similar strategy as in the above proof, we can show that
\begin{eqnarray}
    \det(I-A(\z):i+1,1) &=& \det(I-A(\z):i+m+1,1) \nonumber \\
    &=& z_i(1-z_i)\prod_{j\neq i}(1-z_j^2)\nonumber
\end{eqnarray}
This implies $$\sum_{i} (-1)^{i+1} \det(I-A(\z):i,1)= (1-z_1^2)\cdots(1-z_m^2).$$
\end{proof}
Hence, the proof of Proposition \ref{pro:Fm_growth} follows from Lemmas \ref{lemma:det(I-Az)} and \ref{lemma:minor_sum}.
\end{proof}
We now provide an alternative way of computations. 
\begin{Pro} \label{pro:delta_Fm}
The function $\Delta_{F_m}(\z)$ satisfies
    \begin{equation}
    \Delta_{F_m}(\z) = \displaystyle \frac{1}{1-2 \displaystyle \sum_{i=1}^m \frac{z_i}{1+z_i}}.  \label{eqn:Fm_system}
\end{equation}
\end{Pro}
\begin{proof}
Recall that the elements of $F_m$ are identified with freely reduced words over the alphabet $\Sigma = \{a_1,\cdots , a_m,a_1^{-1},\cdots,a_m^{-1}\}.$ 
$$\Delta_{F_m}(\z) = \displaystyle 1 + \sum_{e\neq g \in F_m} \z^{\wp(g)}$$
where $e$ is the identity in $F_m$, $\z = (z_1,\cdots,z_m), \wp(g) = (|g|_{a_1^{\e}},\cdots,|g|_{a_m^{\e}})$ and 
$\z^{\wp(g)} = z_1^{|g|_{a_1^{\e}}}\cdots z_m^{|g|_{a_m^{\e}}}$.\\
For each $ i = 1,\cdots,m, \e = \pm 1$, let $F_{m,i}^{\e} = \{g \in F_m \backslash \{e\} : \textrm{ the first letter of } g \textrm{ is } a_i^{\e} \}$. Define 
$$\Delta_i^{\e}(\z) = \Delta_{F_{m,i}}^{\e}(\z) = \sum_{e\neq g \in F_{m,i}^{\e}} \z^{\wp(g)}$$
so that we get $$ \Delta_{F_m}(\z) = 1 + \sum_{i,\e} \Delta_i^{\e}(\z).$$
Observe that $$\Delta_i(\z) = z_i + z_i\Delta_i(\z) + z_i \sum_{j\neq i, \e} \Delta_j^{\e}(\z).$$ 
This implies
$$ \Delta_i(\z) = z_i \left( \Delta_{F_m}(\z) - \Delta_i^{-1}(\z)\right).$$
Similarly we can write 
$$ \Delta_i^{-1}(\z) = z_i \left( \Delta_{F_m}(\z) - \Delta_i(\z)\right).$$
Adding $\Delta_i(\z)$ and $\Delta_i^{-1}(\z)$ we get
$$ \Delta_i(\z) + \Delta_i^{-1}(\z) = \displaystyle\left(\frac{2z_i}{1+z_i}\right)\Delta_{F_m}(\z),$$
Hence $$\Delta_{F_m}(\z) = 1 + \displaystyle \sum_i \left(\Delta_i(\z) + \Delta_i^{-1}(\z)\right) = 1 + 2 \sum_{i=1}^m \left(\frac{z_i}{1+z_i}\right)\Delta_{F_m}(\z)$$ 
and we come to \eqref{eqn:Fm_system}.
\end{proof}
We now compare two expressions for $\Delta_{F_m}(\z).$ 
Multiplying numerator and denominator of \eqref{eqn:Fm_system} by $(1+z_1)\cdots (1+z_i)$ we get
\begin{equation}
    \Delta_{F_m}(\z) = \displaystyle \frac{(1+z_1)\cdots (1+z_m)}{\left((1+z_1)\cdots (1+z_m)\right)-2\left((1+z_1)\cdots (1+z_m)\right)\sum_{i=1}^m \frac{z_i}{1+z_i}}\label{eqn:sys_Fm_grow}
\end{equation}
It is easy to see that
$$(1+z_1)\cdots (1+z_m) = 1 + \sum_i z_i + \sum_{i_2 < i_1} z_{i_1}z_{i_2} + \cdots + \displaystyle \sum_{i_1<\cdots< i_l} \prod_{k=1}^l z_{i_k} + \cdots + \prod_{i=1}^m z_i$$
and
$$2\left((1+z_1)\cdots (1+z_m)\right)\sum_{i=1}^m \frac{z_i}{1+z_i}$$
$$= 2\sum_i z_i + 4\sum_{i_2 < i_1} z_{i_1}z_{i_2} + \cdots + 2l\sum_{i_1<\cdots< i_l} \prod_{k=1}^l z_{i_k} + \cdots + 2m \prod_{i=1}^m z_i$$
Substituting above relations in the denominator of \eqref{eqn:sys_Fm_grow} we get $R(\z)$ of \eqref{eqn:Fm_2.1}.
\section{The multivariate growth exponent and the \\condition (CG)}\label{sec:cond_Q}
Let 
\begin{equation}
    \Gamma(\z) = \displaystyle\sum_{\ii \in \N^d} f_{\ii} \z^{\ii} 
    = \displaystyle\sum_{i_1,\cdots,i_d} f_{i_1,\cdots,i_d} z_1^{i_1}\cdots z_d^{i_d} \label{eqn:multigrow_def}
\end{equation}
be a multivariate power series. We denote by $\D$ the interior of the domain of absolute convergence of $\Gamma$, which we assume to be non empty. In particular, $0$ belongs to $\D$.\\

We are going to define for each $\rr \in M_d$ a growth exponent $\varphi(\rr)$ for coefficients $\displaystyle\{f_{\ii}\}_{\ii \in \N^d}$ in the direction of $\rr.$ When $\rr$ is a vector with rational entries, then one can define $\varphi(\rr)$ as
\begin{equation}
    \varphi(\rr) = \displaystyle \limsup_{n \rightarrow \infty} |f_{n\rr}|^{\frac{1}{n}}, \label{eqn:varphi_for_rational_r}
\end{equation}
where only coefficients $n\rr = (nr_1,\cdots,nr_d) \in \N^d$ with integer entries are taken into account. This approach is considered in the book  \cite{melczer2020invitation} and some of its references.\\

To define $\varphi(\rr)$ for arbitrary $\rr \in M_d$ following \cite{quint2002divergence}, we first define the function $\psi(\rr)$. We then hope to be able put $\varphi(\rr) = e^{\psi(\rr)}$, which would then agree with \eqref{eqn:multigrow_rational}.
\begin{Def}\label{def:psi_def}
Let $C \subset \R^d$ be an open (linear) cone. Define
\begin{equation}
    \tau_C = \displaystyle \limsup_{R\rightarrow\infty} \frac{1}{R} ~\log\left(\sum_{\substack{\ii \in C\\ R\leq \norm{\ii} \leq R+1}} f_{\ii}\right) \label{eqn:tau_def}
\end{equation}
and for $\rr \in \R^d, \rr \neq \oo$
\begin{equation}
    \psi(\rr) = \displaystyle\norm{\rr} \inf_{\rr \in C} \tau_C \label{eqn:psi_tau_def}
\end{equation}
where $\inf$ is taken over open cones containing $\rr.$ Also set $\psi(\oo) = \oo.$ 
\end{Def}

We call $\psi(\rr)$ $indicatrice$ of growth. The assumption that $0$ is an interior point of the domain of absolute convergence of $\Gamma$ ensures that $\psi$ can not take the value $+\infty$, but it may take the value $-\infty.$ \\

Recall that the $\relog$ map is a map $\relog: \Co_*^d \rightarrow \R^d,$ where $\Co_* = \Co\backslash\{0\}$ given by 
$$\relog(\z) = (\log|z_1|,\cdots,\log|z_d|).$$
In future, we will use the notation $\E = \R^d$ and $\E^*$ for its dual space i.e. the space  of continuous linear functionals $\theta:\E \rightarrow \R$ with a natural identification of $\E^*$ with $\E$ via the standard inner product 
$$\langle \x, \theta \rangle = \theta(\x) = \displaystyle\sum_{i=1}^d x_i\theta_i,$$ where $\x = (x_1,\cdots,x_d) \in \E$, $\theta = (\theta_1,\cdots,\theta_d) \in \E^*.$

The height function $h_{\rr} (\z)$ for $\z \in \E$ is the function 
$$h_{\rr}(\z) = - \displaystyle\sum_{i=1}^d r_i\log|z_i| = - \langle \rr, \relog(\z)\rangle$$
if $\rr \in M_d\cap\Q^d$ is a rational vector then for \eqref{eqn:varphi_for_rational_r} there is an upper bound
\begin{equation}
   \varphi(\rr) = \displaystyle\limsup_{n \rightarrow \infty} |f_{n\rr}|^{\frac{1}{n}} \leq |z'_1|^{-r_1}\cdots |z'_d|^{-r_d} = e^{-h_{\rr}(\z')} \label{eqn:varphi_interms_hr}
\end{equation}
which hold for every point $\z'$ in the closure of $\overline{\D}$ of $\D$, as shown for instance in (\cite{melczer2020invitation}, formula (5.15)). We shall provide sufficient conditions for replacement of the inequality in \eqref{eqn:varphi_interms_hr} by the equality for properly chosen $\z'$ and extend the definition of $\varphi(\rr)$ to irrational $\rr$ by $\varphi(\rr) = e^{\psi(\rr)}$. The function $\varphi(\rr)$ will be called the multivariate  growth exponent.\\

For what follow it will be convenient to associate with the power series $\Gamma(\z)$ a Radon measure $\nu = \nu_{\Gamma}$ defined on Borel subsets $S$ of $\E$ by 
\begin{equation}
    \nu(S) = \displaystyle \sum_{\ii \in S} f_{\ii} \label{eqn:radon}
\end{equation}
We shall now define the condition (CG).
\begin{Def}\label{def:condition_Q}
(Condition (CG)) Following \cite{quint2002divergence} we say that $\Gamma(\z)$ has a concave growth of coefficients if there are $a,b,c > 0$ such that for all $\x, \y \in \E$ 
\begin{equation}
    \nu\left(B_{\x+\y}(a)\right) \geq c \nu\left(B_{\x}(b)\right)\nu\left(B_{\y}(b)\right) \label{eqn:concavity_Q_condition}
\end{equation}
where $B_{\x}(a)$ is the ball of radius $a$ with the center at $\x \in \E$ for the norm $\norm{.}$.
\end{Def}
A large class of power series $\Gamma(\z)$ satisfying the condition (CG) are generating series of regular languages satisfying the condition (E), as explained in section \ref{sec:cond_E}.\\

Given a Radon measure $\nu$ satisfying condition (CG) one defines for a given open cone $C \subset \E$
\begin{eqnarray}
    \tau_{\nu, C} = \displaystyle \limsup_{R \rightarrow \infty} \displaystyle\frac{1}{R} \log \nu \biggl( \Bigl(B_{\oo}(R+1)\smallsetminus B_{\oo}(R)\Bigr)\cap C \biggr), \nonumber
\end{eqnarray}
and $$ \psi_{\nu}(\x) = \norm{\x} \inf_{\x \in C} \tau_{\nu, C} $$
for $\x \in \E$, with $ \psi_{\nu}(\oo) = \oo.$  Then assuming that $\tau_{\nu} < \infty$, where 
\begin{equation}
    \tau_{\nu} = \displaystyle \sup_{\x \in \E, \norm{\x} = 1} \psi_{\nu}(\x) \label{eqn:tau_as_psi_nu}
\end{equation}
we know from (\cite{quint2002divergence}, Lemma 3.1.7)  that the function $\psi_{\nu} : \E \rightarrow \R\cup \{-\infty\}$ is upper semi-continuous. Obviously it is positively homogeneous  i.e. for any $t>0$
$$\psi_{\nu}(t\x) = t\psi_{\nu}(\x).$$
Finally, if the Condition (CG) holds, then $\psi_{\nu}(\x)$ is concave i.e.
$$\psi_{\nu}(\x+\y) \geq \psi_{\nu}(\x) + \psi_{\nu}(\y), \textrm{ for } \x,\y \in \E.$$
Observe that in our situation, because the measure $\nu = \nu_{\Gamma}$ is supported only on the $\N^d \subset \R_{\geq 0}^d $ part of $\E,$ 
$$\psi_{\nu}(\x) = - \infty, \textrm{ for } \x \notin \R_{\geq 0}^d.$$

Let $\Omega = \relog(\D\cap\mathbb C^d_*) \subset \E$ (recall that $\D \subset \Co^d$ is the interior of the domain of absolute convergence of \eqref{eqn:multigrow_def}). It is well known (\cite{melczer2020invitation}, Proposition 3.4) that $\Omega$ is convex. Let $\psi = \psi_{\Gamma}$ be the function given by Definition \ref{def:psi_def}. 
\begin{The}\label{the:psi_as_support}
Let $\Gamma(\z)$ be a power series given by the Equation \eqref{eqn:multigrow_def} with non-negative coefficients $f_{\ii}, \ii \in \N^d$ such that $f_{\ii}$ satisfy the concavity condition (CG) and $\tau_{\nu} < \infty$ where $\nu = \nu_{\Gamma}$ is given by the Equation \eqref{eqn:radon} and $\tau_{\nu}$ is given by \eqref{eqn:tau_as_psi_nu}. Let $\D \subset \Co^d$ be the interior of the set of points of absolute convergence of $\Gamma(\z),$  
$\Omega = \relog (\D\cap\mathbb C^d_*)$ and $\overline{\Omega}$ be the closure of $\Omega.$ Then $- \psi_{\Gamma}(\x)$ is the support function of the closure $\overline{\Omega}$ and for $\x \in \R_{\geq 0}^d$ 
\begin{eqnarray}
    \psi_{\Gamma}(\x) &=& \displaystyle \inf_{\theta \in -\overline{\Omega}} \langle \x, \theta \rangle \label{eqn:psi_as_inf} \\
               &=& \displaystyle \inf_{\theta \in -\partial\overline{\Omega}} \langle \x, \theta \rangle \label{eqn:psi_as_inf_on_boundary}
               \mbox{ if }\Omega\neq\mathbb R^d.
\end{eqnarray}
\end{The}
\begin{proof}
Because of the Condition (CG), the function $-\psi_{\Gamma}(\x)$ is a lower semi-continuous, convex and positively homogeneous. Hence, it is a support function of a one and only one closed convex subset $S \subset \E^*$ given by  
\begin{eqnarray}
            S  &=& \displaystyle \left\{\theta \in \E^* : -\psi_{\Gamma}(\x) \geq \langle \x, \theta \rangle, \forall \x \in \E\right\} \nonumber \\
               &=& \displaystyle \left\{\theta \in \E^* : \psi_{\Gamma}(\x) \leq -\langle \x, \theta \rangle, \forall \x \in \E\right\} \label{eqn:psi_with_theta} 
\end{eqnarray}
The domain of absolute convergence of $\Gamma(\z)$ is determined by the condition that the integral given below in the Equation \eqref{eqn:multigrow_as_integral} is convergent.
\begin{eqnarray}
    \Gamma(|z_1|,\cdots,|z_d|)  &=& \displaystyle \sum_{\ii} f_{\ii} e^{\sum_{k=1}^d i_k \log|z_k|} \nonumber \\
               &=& \displaystyle \int_{\E} e^{\langle \x, \theta \rangle} d\nu_{\Gamma}(\x) \label{eqn:multigrow_as_integral}
\end{eqnarray}
where $\theta = \relog(\z), \z = (z_1,\cdots,z_d).$\\

Consider the set
$$\Delta^{\circ}_{\Gamma} = \left\{\theta \in \E^* : -\langle \x, \theta \rangle > \psi_{\Gamma}(\x), \forall \x \in \E\backslash\{\oo\}\right\}.$$
Using (\cite{quint2002divergence}, Lemma 3.1.3), we conclude that if $\theta \in \Delta^{\circ}_{\Gamma}$, then the integral \eqref{eqn:multigrow_as_integral} converges while if there is $\x \in \E\backslash\{\oo\}$ with $-\langle\x,\theta\rangle < \psi_{\Gamma}(\x),$ then it diverges. The closure $\Delta_{\Gamma} = \overline{\Delta^{\circ}_{\Gamma}}$ is the set
$$ \Delta_{\Gamma} = \left\{\theta \in \E^* : -\langle \x, \theta \rangle \geq \psi_{\Gamma}(\x), \forall \x \in \E\right\}$$
which is the set $S$ given by the Equation \eqref{eqn:psi_with_theta} for whom $-\psi_{\Gamma}(\x)$ is the support function. Thus
\begin{eqnarray}
    -\psi_{\Gamma}(\x) &=& \displaystyle \sup_{\theta \in S} \langle \x, \theta \rangle \nonumber \\
    \psi_{\Gamma}(\x)  &=& \displaystyle -\sup_{\theta \in S} \langle \x, \theta \rangle \nonumber \\
                &=& \displaystyle \inf_{\theta \in -S} \langle \x, \theta \rangle \nonumber
\end{eqnarray}
Also, we conclude that $$S = \relog(\overline{\D}) = \overline{\Omega}. $$ This leads us to the Equation \eqref{eqn:psi_as_inf} and eventually to \eqref{eqn:psi_as_inf_on_boundary} as the maximum of a linear functional taken over the closed convex set $S$ is achieved on the boundary $\partial S$ of $S$ or equals to $-\infty.$  
\end{proof}
Following is the immediate consequence of Theorem \ref{the:psi_as_support}.
\begin{cor}\label{cor:psi_as_support}
    \begin{equation}
        \psi(\x) = \displaystyle \inf_{\z \in \partial\D} \left(-\sum x_i\log|z_i| \right) \label{eqn:psi_as_inf_on_dabaD}
    \end{equation}
\end{cor}

{\bf Example}\\
Let $d =2, \Sigma = \{a_1,a_2\}$ and $L = \Sigma^*$ be the language of a free monoid. Then  
$$\Gamma_L(\z) = \displaystyle\frac{1}{1-z_1-z_2} = \displaystyle \sum_{n=0}^{\infty}\sum_{i=0}^{n} \binom ni z_1^i z_2^{n-i}.$$ 
It is obvious that the Condition (CG) holds and the direct computation based on the use of Stirling's formula or Theorem \ref{the:psi_as_support} show that for $\rr \in M_2$ the function $\psi_L(\x)$ is the Shannon's entropy $\mathbf{H}(\rr)$. i.e.
$$\psi_L(\rr) = \mathbf{H}(\rr) = - r_1\log r_1 - r_2\log r_2, \rr \in M_2.$$
More examples will be discussed in the following two sections.

\section{Multivariate growth exponent in the case of the free group} \label{sec:indicator_F_m}

In the section \ref{sec:Fm_gr}, we got two expressions for the multivariate growth series $\Gamma_{F_m}(\z)$ for the language $L_{F_m}$ of freely reduced words in the alphabet $\{a_1,\cdots,a_m,a_1^{-1},$\\$\cdots,a_m^{-1}\}.$ \\

Consider the case $m=2.$ Then 
\begin{equation}
    \Delta_{F_2}(\z) = \displaystyle \frac{(1+z_1)(1+z_2)}{1-z_1-z_2-3z_1z_2} \label{eqn:multigrow_F2}
\end{equation}
We use the notations $\D, \Omega,$ $\relog$ of Section \ref{sec:cond_Q}. First, we are mainly interested in understanding the shape of the set $\Omega = \relog(\D).$
To get a view of the real slice of the domain of the absolute convergence of the power series \eqref{eqn:multigrow_F2}, we observe that the domain $\D$ is described as
$$\D = \{ (z_1,z_2) \in \Co^2 : |z_1|+|z_2|+3|z_1||z_2| < 1\}$$
The real slice of the curve 
$$H(z_1,z_2) = f_1(\z) = 1-z_1-z_2-3z_1z_2 = 0$$ is presented by hyperbola in Figure \ref{subfig:D_F2}.

\begin{figure}[!htb]
\begin{subfigure}[b]{0.45\textwidth}
\centering
    \includegraphics[width=\textwidth]{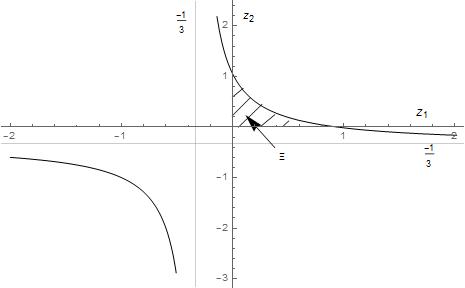}
    \caption{}
   \label{subfig:D_F2}
\end{subfigure}
\hfill
\begin{subfigure}[b]{0.45\textwidth}
\centering
    \includegraphics[width=\textwidth]{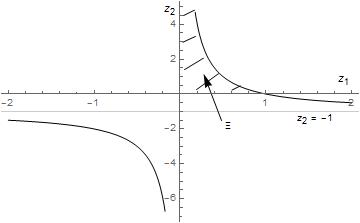}
    \caption{}
    \label{subfig:D_Fib}
\end{subfigure}
\caption{Real part of the Domain $\D$ of $H(z_1,z_2)$ of $L_{F_2}$ and $L_{Fib}$, respectively.}
\label{fig:domain}
\end{figure}

The real slice of $\D$ with $z_1, z_2 \in \R, z_1,z_2 \geq 0$ is presented by the tented area $\Xi.$ The set $-\Omega$ is obtained from $\Xi$ by making the substitution $z_1 = e^{-s}, z_2 = e^{-t}.$ To get the clear picture of the shape of the set $\Omega$ (and hence $-\Omega$) we make use of an interesting notion from algebraic geometry called $amoeba$. 
\begin{figure}[!htb]
\begin{subfigure}[b]{0.475\textwidth}
\centering
    \includegraphics[width=\textwidth]{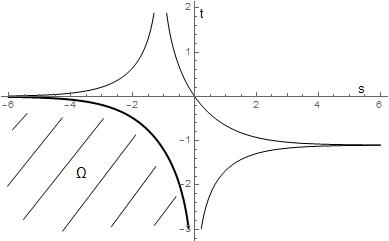}
    \caption{}
   \label{subfig:Amoeba_F2}
\end{subfigure}
\hfill
\begin{subfigure}[b]{0.475\textwidth}
\centering
    \includegraphics[width=\textwidth]{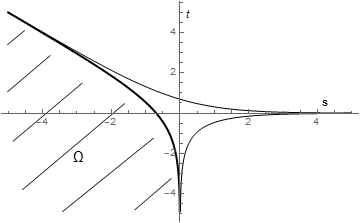}
    \caption{}
    \label{subfig:Amoeba_Fib}
\end{subfigure}
\caption{The $amoeba(f_1)$, $amoeba(f_2)$ along with sets $\Omega$ associated to the language $L_{F_2}$ and the Fibonacci language are presented in \ref{subfig:Amoeba_F2} and \ref{subfig:Amoeba_Fib}, respectively, where $f_1(\z) = 1-z_1-z_2-3z_1z_2$ and $f_2(\z) = 1-z_1-z_1z_2$.}
\label{fig:amoeba}
\end{figure}

Recall, that given a Laurent polynomial $f(\z), \z \in \Co^d,$ the amoeba of f is the set 
$$ amoeba(f) = \left\{ \relog(\z) : \z \in \Co_*^d , f (z) = 0 \right\} \subset \R^d, $$
where $\Co_* = \Co \backslash \{0\}$. The amoeba’s complement is $amoeba(R)^c = \R^d\backslash amoeba(R).$\\

The following result follows from Gelfand, Kapranov, and Zelevinsky \cite[Chap. 6, Prop. 1.5 and Cor. 1.6]{gelfand2008discriminants}. 
\begin{Pro}\label{pro:Gelfand}
If $f(\z)$ is a Laurent polynomial then all connected components of the complement $amoeba(f)^c$ are convex subsets of $\R^d.$ These real convex subsets are in bijection with the Laurent series expansions of the rational function $\frac{1}{f (z)}.$ When $\frac{1}{f (z)}$ has a power series expansion, then it corresponds to the component of $\R^d \backslash amoeba(f)$ containing all points $(-N,\cdots,-N)$ for $N$ positive and sufficiently large.
\end{Pro}

The techniques of drawing amoeba are well developed. The amoebas of polynomials $f_1(\z) = 1-z_1-z_2-3z_1z_2$ and $f_2(\z) = 1-z_1-z_1z_2$ that correspond to the cases of language $L_{F_2}$ of freely reduced words of $F_2$ and Fibonacci language, respectively, are presented in the Figure \ref{fig:amoeba}.\\

Looking at the shape of $-\Omega$ presented in Figure \ref{fig:_Omega}, in the case of $L_{F_2}$, we see that for each $\rr \in M_2$ with positive entries there is a tangent line to the boundary $\partial(-\Omega)$, is orthogonal to $\rr$ and hence, the infimum of the linear form \eqref{eqn:psi_as_inf_on_boundary} is achieved. 
\begin{figure}[!htb]
\begin{subfigure}[b]{0.5\textwidth}
\centering
    \includegraphics[width=\textwidth]{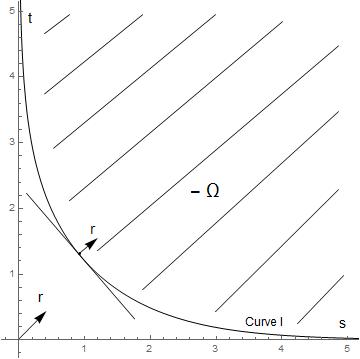}
    \caption{}
   \label{subfig:_Omega_F2}
\end{subfigure}
\hfill
\begin{subfigure}[b]{0.35\textwidth}
\centering
    \includegraphics[width=\textwidth]{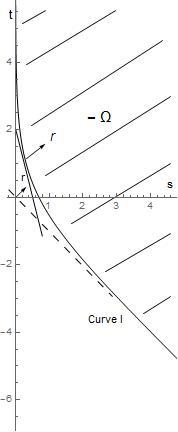}
    \caption{}
    \label{subfig:_Omega_Fib}
\end{subfigure}
\caption{The sets $-\Omega$ presented in \ref{subfig:_Omega_F2} and \ref{subfig:_Omega_Fib} are associated with languages $L_{F_2}$ and $L_{Fib}$, respectively.}
\label{fig:_Omega}
\end{figure}
To compute $\psi(\rr)$ we apply a standard method of Lagrange multipliers. \\

In coordinates $(s,t) \in \R^2,$ $z_1 = e^{-s}, z_2 = e^{-t}$, the boundary of $-\Omega$ is a curve $l$ given by the equation 
\begin{equation}
    1- e^{-s} - e^{-t} - 3e^{-s-t} = 0 \label{eqn:F2_in_minus_s_t}
\end{equation}
or
\begin{equation}
    e^{s+t}- e^{s} - e^{t} - 3= 0 \label{eqn:F2_in_plus_s_t} 
\end{equation}
Let $\rr = (p,q) \in M_2.$ We have to minimize $ps+qt $ when $(s,t) \in l.$ The associated Lagrange function is
    $$ \Phi(s,t,\lambda) = ps+qt - \lambda(e^{s+t}- e^{s} - e^{t} - 3).$$
    Equating partial derivatives to zero we obtain
    \begin{eqnarray}
    \displaystyle\frac{\partial\Phi}{\partial s} &=&  p - \lambda \left( e^{s+t} - e^s\right) = 0 \implies p = \lambda\left( e^{s+t} - e^s\right) \nonumber\\
    \displaystyle\frac{\partial\Phi}{\partial t} &=&  q - \lambda \left( e^{s+t} - e^t\right) = 0\implies q = \lambda \left( e^{s+t} - e^t\right) \nonumber\\
    \displaystyle\frac{\partial\Phi}{\partial \lambda} &=& e^{s+t}- e^{s} - e^{t} - 3 = 0. \nonumber
    \end{eqnarray}
This gives $$\rho = \displaystyle \frac{p}{q} = \frac{1-e^{-t}}{1-e^{-s}} \implies e^{-t} = 1-(1-e^{-s})\rho$$ 
Substituting the value of $e^{-t}$ in \eqref{eqn:F2_in_minus_s_t} we get
\begin{eqnarray}
1-e^{-s}-\left(1-(1-e^{-s})\rho\right)-3e^{-s}\left(1-(1-e^{-s})\rho\right) &=& 0 \nonumber\\
-e^{-s}+\rho-\rho e^{-s}-3e^{-s}+3\rho e^{-s}-3\rho e^{-2s} &=& 0 \nonumber\\
-3\rho e^{-2s} +\left(2\rho-4\right)e^{-s}+\rho &=& 0 \nonumber 
\end{eqnarray}
Substituting $x = e^{-s}$ in the above quadratic and then solving it gives
$$x_{1,2} = \displaystyle \frac{2\rho \pm \sqrt{1-\rho+\rho^2}}{6\rho}$$
We choose positive sign (i.e. +) because we know due to \cite{GQ23} that the function is real analytic. Re-substituting  $x = e^{-s}$ and $\rho = \displaystyle \frac{p}{q}$ we get
$$e^{-s} = \displaystyle \frac{p-2q+2\sqrt{p^2-pq+q^2}}{3p},$$
$$e^s =  \displaystyle \frac{2q-p+2\sqrt{p^2-pq+q^2}}{p}$$
and by symmetry, 
$$e^t =  \displaystyle \frac{2p-q+2\sqrt{p^2-pq+q^2}}{q}.$$
Hence,
$$s = \displaystyle\log\left(\frac{2q-p+2\sqrt{p^2-pq+q^2}}{p}\right),~~ t = \displaystyle\log\left(\frac{2p-q+2\sqrt{p^2-pq+q^2}}{q}\right),$$
and
$$\psi_{F_2}(\rr) = p \displaystyle\log\left(\frac{2q-p+2\sqrt{p^2-pq+q^2}}{p}\right)+ q\displaystyle\log\left(\frac{2p-q+2\sqrt{p^2-pq+q^2}}{q}\right)$$
or
$$\psi_{F_2}(\rr) = \mathbf{H}(\rr) + p \displaystyle\log\left(2q-p+2\sqrt{p^2-pq+q^2}\right)+ q\displaystyle\log\left(2p-q+2\sqrt{p^2-pq+q^2}\right),$$
where $\mathbf{H}(\rr) = -p\log p - q \log q.$ 
See Figure \ref{subfig:ind_F2} for the graph of $\psi_{F_2}(\rr) = \psi_{F_2}(p,1-p).$

\section{Multivariate growth in the case of Fibonacci language } \label{sec:indicator_fib}

Fibonacci language $L_{Fib}$ is a language over binary alphabet $\{0,1\}$ consisting of words that do not contain $11$ as a subword i.e. the word $11$ is forbidden. It is one of the important examples associated with subshifts of finite type \cite{lind1995introduction}. The reason why Fibonacci name is associated to it is coming from the fact that the number of words of length $n$ in $L_{Fib}$ is equal to the $(n+2)th$ Fibonacci number $\beta.$
$$\beta = \displaystyle \frac{1}{\sqrt{5}} \left(\lambda_1^{n+2}-\lambda_2^{n+2}\right),$$ where $\lambda_1 = \displaystyle\frac{1+\sqrt{5}}{2}$ and $\lambda_2 = \displaystyle\frac{1-\sqrt{5}}{2}$ are eigenvalues of matrix 
$$A = \left(\begin{array}{cc}
  1 & 1 \\
  1 & 0 \\
\end{array}\right)$$
which is the matrix of transitions of the Fibonacci subshift. At the same time $A$ is the adjacency matrix of the automaton $\A_{Fib}$ given in the Figure \ref{fig:fibonacci} 
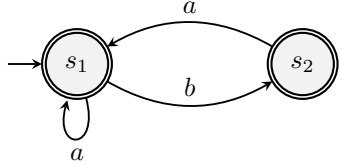
\begin{figure}[!htb]
\centering
    \begin{tikzpicture}
        \node[state, initial, accepting] (q1) {$s_1$};
        \node[state, accepting, right of=q1] (q2) {$s_2$};
       
        \draw[->]   (q1) edge[bend right, above] node{$b$} (q2)
                (q1) edge[loop below] node{$a$} (q1)
                (q2) edge[bend right, above] node{$a$} (q1);
    \end{tikzpicture}
    \caption{Moore diagram of Fibonacci automaton $\A_{Fib}$}
    \label{fig:fibonacci}
\end{figure}
which accepts the language $L_{Fib}$, if instead of $\{0,1\}$ we use the alphabet $\{a,b\}.$ An easy computation of growth series gives function 
$$\Gamma_{Fib}(\z) = \displaystyle\frac{1+z_2}{1-z_1-z_1z_2}$$ mentioned before Definition \ref{def:condition_E}. In order to compute $\psi$, observe from Figure \ref{fig:fibonacci} that the Condition (CG) holds as the automaton $\A_{Fib}$ is ergodic. 

\begin{Pro}\label{pro:fib_psi}
The indicatrice $\psi_{Fib}$ viewed as a function of $p, 0 < p < 1$ where $\rr = (p,1-p) \in M_2$ is the direction vector, is given by

\begin{eqnarray}
      \psi_{Fib}(\rr) = \left\{ 
      \begin{array}{lr}
        p \log \displaystyle \left(\frac{p}{2p-1}\right) + (1-p)\log \displaystyle\left( \frac{2p-1}{1-p}\right)   & \textnormal{ if } p \geq \displaystyle\frac{1}{2} \label{eqn:psi_fib} \\
         -\infty   & \textnormal{ if } p < \displaystyle\frac{1}{2} \nonumber
      \end{array} \right.
\end{eqnarray}
\end{Pro}
The graph of $\psi_{Fib}(\rr)$ on $[\frac{1}{2},\infty)$ is shown in the Figure \ref{subfig:ind_Fib}.
\begin{proof} As before, we switch to the variables $x,y$ instead of $z_1,z_2,$ respectively. The amoeba of $f_2(x,y) = 1-x-xy$ and the set $-\Omega$ are shown in the Figures \ref{subfig:Amoeba_Fib} and \ref{subfig:_Omega_Fib}, respectively.\\

We provide two explanations why $\psi_{Fib}(p) = -\infty$ when $0 < p < \displaystyle\frac{1}{2},$ first, algebraic and second, geometric. The power series expansion of $(1-x-xy)^{-1}$ is
\begin{eqnarray}
\displaystyle\frac{1}{1-x-xy} &=& \displaystyle\sum_{n=1}^{\infty}\sum_{i=0}^n \binom ni x^{n-i} (xy)^{i} \nonumber\\
&=& \displaystyle\sum_{n=1}^{\infty}x^n\sum_{i=0}^n \binom ni y^{i} \nonumber
\end{eqnarray}
Hence,
\begin{eqnarray}
\displaystyle\frac{1+y}{1-x-xy} &=& \displaystyle\sum_{n=1}^{\infty}x^n\sum_{i=0}^n \binom ni y^{i} + \displaystyle\sum_{n=1}^{\infty}x^n\sum_{i=0}^n \binom ni y^{i+1} \nonumber\\
&=& \displaystyle\sum_{n=1}^{\infty}x^n\left\{1+y+\sum_{i=0}^n \left[ \binom ni + \binom{n}{i-1} \right] y^{i} \right\} \nonumber
\end{eqnarray}
and we see that coefficients of the power series corresponding to the indices $(n,i)$ with $i > n+1$ are zero. Hence, any direction $\rr = (p,q)$ with $\frac{p}{q} > 1$ gives value $-\infty$ to $\psi_{Fib}.$ As for the geometric explanation let us use Figures \ref{subfig:Amoeba_Fib} and \ref{subfig:_Omega_Fib}. The domain $-\Omega$ is shown in Figure \ref{subfig:_Omega_Fib} and it is bounded by the curve $l$ given by the equation 
\begin{equation}
    1-e^{-s}-e^{-s-t} = 0 \label{eqn:Fib_denom_in_e}
\end{equation}
\begin{itemize}
    \item Case (i) If $p < \frac{1}{2},$ then $\frac{1-p}{p} > \frac{1}{2}$ and lines $$ps+(1-p)t =c$$ intersect $-\Omega$ for arbitrary $c \in \R.$ Hence,
    $$\psi_{Fib}(p) = \displaystyle \inf_{(s,t) \in -\Omega} \left(ps + (1-p)t\right) = -\infty$$
    \item Case (ii) If $p > \frac{1}{2},$ then $\frac{1-p}{p} < \frac{1}{2}$ and there is a unique value of $c$
    such that the line $$ps+(1-p)t =c$$ is tangent to the curve $l.$ The coordinates $(s_0,t_0)$ of the tangent point $P$ gives a minimum value to the linear form $ps+(1-p)t$ when $(s,t) \in -\Omega.$ To find it again, we again apply the method of Lagrange. We denote $q = 1-p$ and rewrite Equation \eqref{eqn:Fib_denom_in_e} as 
    $$e^{s+t}-e^t-1=0$$
    The associated Lagrange function is
    $$ \Phi(s,t,\lambda) = ps+qt - \lambda(e^{s+t}-e^t-1)$$
    and the corresponding system of equation is
    \begin{eqnarray}
    \displaystyle\frac{\partial\Phi}{\partial s} &=&  p - \lambda e^{s+t} = 0 \nonumber\\
    \displaystyle\frac{\partial\Phi}{\partial t} &=&  q - \lambda e^{s+t} + \lambda e^t = 0 \label{eqn:fib_system}\\
    \displaystyle\frac{\partial\Phi}{\partial \lambda} &=& e^{s+t}-e^t-1 = 0 \nonumber
    \end{eqnarray}
    From the first equation in \eqref{eqn:fib_system}, we get $e^{s+t} = \frac{p}{\lambda}.$ from the second equation we get
    $q = p + \lambda e^t,$ and from the third equation we get $\frac{p}{\lambda} - e^t -1 =0.$ Hence, 
    $$e^t = \displaystyle \frac{p-q}{\lambda} = \frac{2p-1}{\lambda} $$
    But $\lambda = 1-p = q$. This gives,
    $$e^t = \displaystyle \frac{2p-1}{1-p} \textnormal{ and hence, } e^s = \frac{p}{2p-1}.$$
    The last equation determine the point $(s_0,t_0).$ Making substitution in \eqref{eqn:psi_as_inf_on_boundary} we obtain \eqref{eqn:psi_fib}.
\end{itemize}
\end{proof}

\section{Multivariate growth and LDT}\label{sec:LDT}
The alternative approach for computing $\psi(\rr)$ is via the application of methods of Large Deviation Theory (LDT). Here we discuss a special case related to languages associated with subshifts of finite type. Let us first recall the basic facts about subshifts of finite type (SFT). For more details see \cite{lind1995introduction}. Let $\Sigma = \{a_1,\cdots,a_d\}$ be a finite alphabet and $\Sigma^{\Z}$ be a space of two-sided infinite sequences over $\Sigma$ indexed by integers. $\Sigma^{\Z}$ is supplied by a product topology  and is homeomorphic to a Cantor set when $d \geq 2.$ The shift map $U : \Sigma^{\Z} \rightarrow \Sigma^{\Z}$ is the homeomorphism given by $$(Uw)_n = w_{n+1}, ~w = (w)_{n=-\infty}^{\infty} \in \Sigma^{\Z}.$$
A closed $U$-invariant subset $X \subset \Sigma^{\Z}$ is a subshift. Let $L_X\subset \Sigma^*$ be a language of subshift consisting of (finite) words that appear as a subwords of $w \in X.$ A subshift $X$ is said to be subshift of finite type if there is a finite subset ${\bf F} \subset \Sigma^*$ (set of forbidden words) such that 
$X = X_{\bf F}$ consist of sequences $w \in \Sigma^{\Z}$ that do not contain forbidden subwords. It is obvious that $X_{\bf F}$ is closed and $U$-invariant. For instance, in the case $\Sigma = \{0,1\}$ and ${\bf F} = \{11\}$ we get the Fibonacci subshift. Alternative way to define subshifts of finite type is via the graph $\Xi =(V,E)$ (in fact directed multi-graph i.e. loops and multiple edges are allowed) or equivalently, via the adjacency matrices $A = (a_{ij})$ of the size $|V|\times|V|$ whose rows and columns correspond to vertices of the graph and 
$a_{ij}, i,j \in V$ is a non-negative integer equal to the number of edges joining vertex $i$ with vertex $j.$ For instance, the Fibonacci subshift the matrix is 
$$ A = \left(\begin{array}{cc} 1 & 1 \\ 1 & 0 \\\end{array}\right)$$  
and the graph $\Xi$ is given by the Figure \ref{fig:gamma_fibonacci} and is similar to the diagram of the automaton from Figure \ref{fig:fibonacci}. 
\begin{figure}[!htb]
\centering
    \begin{tikzpicture}
        \node[state, ] (q1) {};
        \node[state,  right of=q1] (q2) {};
       
        \draw[->]   (q1) edge[bend right, above] node{$1$} (q2)
                (q1) edge[loop below] node{$0$} (q1)
                (q2) edge[bend right, above] node{$0$} (q1);
    \end{tikzpicture}
    \caption{The graph $\Xi$ of the Fibonacci subshift}
    \label{fig:gamma_fibonacci}
\end{figure}
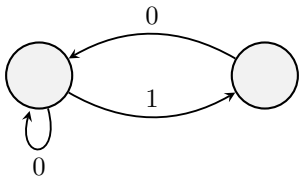

Another example is the free group $F_m$, the subshift $X$, $m \geq 2$ with alphabet $\Sigma = \{a_1,\cdots,a_m,a_1^{-1},\cdots,a_m^{-1}\}$ and 
$$A = \left(
\begin{array}{cc}
 J_{m\times m} & K_{m\times m} \\
 K_{m\times m} & J_{m\times m}
\end{array}\right),$$
where $J$ is matrix with all entries $1$, $K = J-I$ and $I$ is an $m\times m$ identity matrix. The language $L_X$ consists of freely reduced words.\\

It is well known that the language $L_X$ associated with a subshift of finite type is regular. Hence, its multivariate growth series represents a  rational function and the technique of computation of multivariate growth rate described in the previous sections is applicable. Now, we shall see how the results of LDT can be used for the same goal.\\

To make one more step towards LDT, let us recall some other important notions related to SFT. The SFT $(U,X)$ is \emph{irreducible} if the graph $\Xi_X$ is strongly connected (i.e. for any vertices in graph, there is a path connecting them). In this case, the associated matrix $A$ is called irreducible. The Perron-Frobenius theorem states that the irreducible matrix $A$ with non-negative entries (like in our case) has a simple eigenvalue $\rho = \rho(A)$ (called Perron-Frobenius eigenvalue) such that any other eigenvalue $\lambda$ satisfies $|\lambda| < \rho.$ Also there are two vectors $\uu, \vv \in \R^d$ satisfying $$A\vv = \rho\vv, \uu^tA = \rho\uu^t,$$ where $\uu^t$ is the transpose of the column vector $\uu.$ These vectors are unique up to scalar factor. Assume that $A$ is primitive matrix. Then the Perron-Frobenius triple $(\rho,\uu,\vv)$ (consisting of vectors $\uu,\vv > 0$ such that $A\vv = \rho\vv, \uu^tA = \rho\uu^t,$ and normalized by the condition $\langle \uu^t , \vv \rangle = 1$) gives the information about the powers $A^n$ of $A$:
$$\displaystyle\lim_{n\rightarrow\infty} \frac{1}{\rho^n}A^n = \vv\cdot\uu^t,$$ where $\uu^t$ is the transpose of $\uu.$ See (\cite{lind1995introduction}, Theorem 4.5.12).\\

Next, we need the notion of Markov measure on $\Sigma^{\Z}.$ Given a a stochastic $d\times d$ matrix $$P = (p_{ij}), \textrm{ where } p_{ij} \geq 0 \textrm{ and } \displaystyle\sum_{j=1}^dp_{ij}=1, i = 1,\cdots,d, $$
and a stationary probability row vector $p = (p_1,\cdots,p_d)$, $pP = p,$ one can define a Borel probability measure $\mu = \mu_P$ on $\Sigma^{\Z}$ by 
$$\mu\left([\omega_0,\cdots,\omega_n]\right) = p_{\omega_0}\displaystyle\prod_{i=0}^{n-1}p_{\omega_i,\omega_{i+1}},$$ where $[\omega_0,\cdots,\omega_n]$ is a cylinder subset of $\Sigma^{\Z}$ consisting of all $\omega \in \Sigma^{\Z}$ with the prescribed entries $\omega_0,\cdots,\omega_n$ at coordinates $0,1,\cdots,n.$ (such sets generate the sigma-algebra of Borel subsets and hence, values of $\mu$ on them determine $\mu$ completely). The Perron-Frobenius eigenvalue of $P$ is $1,$ vector $p$ exists and is unique if $P$ is irreducible. The measure $\mu_P$ is shift-invariant and the system $\left(\Sigma^{\Z},U,\mu\right)$ is ergodic.\\

One of the main results of theory of SFT is a statement (assuming irreducibility of subshift $X$) on the existence and uniqueness of probability measure $\eta = \eta_X$ of maximal entropy. Not getting into the details we just mentioned that if $\vartheta$ is $U$-invariant probability measures on $\Sigma^{\Z},$ then the Kolmogorov-Sinai entropy $h(\vartheta)$ can be defined. Then 
$$h(\eta) = \displaystyle\max_{\vartheta} h(\vartheta),$$
where maximum is taken over all $U$-invariant probability measures supported on $X.$ The measure $\eta$ is called \emph{Parry} measure. It is a Markov type measure determined by a stochastic matrix $P = (p_{ij}),$ with 
\begin{equation}
    p_{ij} = \displaystyle\frac{1}{\rho} a_{ij}\frac{v_j}{v_i}, \label{eqn:perry}
\end{equation}
where $A = (a_{ij})$ and $A\vv = \rho\vv.$ For such measure we have 
\begin{equation}
    \mu\left([i,x_1,\cdots,x_{n-1},j]\right) = \displaystyle\frac{u_iv_j}{\rho^n} \label{eqn:mu_cylinder1}
\end{equation}
See \cite{parryMR161372_1964} for further details. In fact, for $\eqref{eqn:perry}$ and \eqref{eqn:mu_cylinder1} we have to assume that $a_{ij} \in \{0,1\}$ (i.e. graph $\Gamma$ does not have multiple edges). Meanwhile observe that, every SFT can be coded in such a way that the matrix $A$ will have only entries $\{0,1\}$ \cite{lind1995introduction}. Starting from this moment we assume that the measure of maximal entropy is associated with SFT. \\

Now, we recall some basic notions of LDT, namely, the notion of the rate function $I$ and the Large Deviation Principal (LDP). A rate function is a lower semi-continuous function $I: W \rightarrow [0,+\infty]$ defined on a topological space $W$ (for us now $W = \R^d$) such that for each $a \in [0,\infty)$ the level set $$Y_I(a) = \{\x \in W : I(\x) \leq a\}$$ is a closed subset of $W.$ A \emph{good rate function} is a rate function for which all level sets $Y_I(a)$ are compact. A sequence $\{\mu_n\}_{n=1}^{\infty}$ of Borel measures on $W$ satisfies LDP if, for every Borel subset $B \subset W,$
$$- \displaystyle\inf_{\x \in B^{\circ}} I(x) \leq \displaystyle\liminf_{n \rightarrow \infty} \frac{1}{n} \log \mu_n(B) \leq \limsup_{n \rightarrow \infty} \frac{1}{n} \log \mu_n(B) \leq - \inf_{\x \in \overline{B}} I(\x),$$ where $B^{\circ}, \overline{B}$ are,  respectively, the interior and the closure set of the set $B$. \\

Given a finite Markov chain on a set $\Sigma = \{a_1,\cdots,a_d\}$ determined by a stochastic matrix $P = (p_{ij})_{i,j}^d$ and a function $f : \Sigma \rightarrow \R^d, d \geq 1$ one can consider for each $x = (x_i)_{i=1}^{\infty} \in \Sigma^{\N}$ the empirical means
$$Z_n(x) = \displaystyle\frac{1}{n}\sum_{i=1}^nf\left(x_i\right)$$ and the corresponding distributions $\mu_n$ (discrete measures in $\R^d$). We assume that the random process is a Markov process given by the matrix $P$ and stationary vector $p$, $pP = p.$\\

Associate with every $\y \in \R^d$ a non-negative matrix $\Pi(\y)$ whose elements are 
$$ \pi_{ij}(\y) = p_{ij} e^{\langle \y,f(j) \rangle}, i,j \in \Sigma.$$
$\Pi(\y)$ is a matrix with non-negative entries and it is irreducible if and only if $P$ is irreducible. Let $\rho(\Pi(\y))$ be the Perron-Frobenious eigenvalue of $\Pi(\y).$
The Theorem 3.1.2 of \cite{zeitouni1998large} states that the empirical mean $Z_n$ (or corresponding distributions $\mu_n$) satisfies the LDP with the convex and good rate function 
$$I(\z) = \displaystyle \sup_{\y \in \R^d}  \left\{ \langle \y,\z \rangle - \log\rho\left(\Pi(\y)\right)\right\}, \z \in \R^d.$$
There is a version of this result due to Sanov which is more suitable for our goals. Let $f: \Sigma \rightarrow \R^d, |\Sigma| = d$ be a function such that
$$ f(a_i) = \left(
\begin{array}{c}
 0 \\
\vdots\\
1\\
\vdots\\
0
\end{array}\right)$$
where $1$ is at $i$-th position and $i = 1,\cdots,d.$\\

The following alternative description of $I$ holds (\cite{zeitouni1998large}, Theorem 3.1.6).
\begin{eqnarray}
     I(\rr) = \left\{ 
     \begin{array}{lr}
        \displaystyle \sup_{\uu \gg \oo} \sum_{j=1}^d r_j \log \left(\frac{u_j}{(\uu P)_j}\right)  &  \rr \in M_d \label{eqn:sanov}\\
         \infty & \rr \notin M_d \nonumber
     \end{array}\right.
\end{eqnarray}
The supremum is taken over the strictly positive vectors $\uu$ i.e. $u_i > 0$ for all $i$.\\

Now we have all needed to describe the connection between $\psi$ and $I.$ We assume that the language $L\subset\Sigma^*$ is a language determined by the automaton $\A$ with the property that for each state $q \in Q$ all incoming edges are labeled by the same symbol (like in the examples given by the Figure \ref{fig:ergodic_and_nonergodic} or \ref{fig:fibonacci}). The indicatrice of growth as before is denoted by $\psi(\rr).$ We assume that $\A$ is ergodic (and hence, condition (CG) satisfied). Let $A$ be the adjacency matrix of $\A$ and $X_A$ be the corresponding subshift. Let $L_1$ be a language associated with $X_A$. In described situation we identify $\Sigma$ with $Q$ attaching to each state $q \in Q$ symbol $a_q$ (the label of the entering edges). Because $\A$ is ergodic, we get a bijection between $\Sigma$ and $Q.$ After such identification, it become obvious that  $L \subset L_1.$ Because of ergodicity (i.e. irreducibility of $A$) as easily can be shown the indicatrice of growth $\psi(\rr)$ is same for languages $L$ and $L_1.$ \\

Associate with $A = (a_{ij})$ the stochastic matrix $P = (p_{ij}),$ where \\$ p_{ij} = \displaystyle \frac{a_{ij}~ v_j}{\rho~v_i},$ $\rho=\rho(A)$ and $A\vv = \rho\vv,~\vv \gg 0.$ Let $p$ be a stationary probability vector $(pP=p)$ and $\mu = \mu_p$ be a corresponding Markov measure (i.e. Perry measure). From \eqref{eqn:sanov} 
(assuming normalization $\langle p,\vv\rangle = 1$) we know that the measure $\mu$ is almost equidistributed on the cylinder sets $C_w$ determined by the words $w \in L_1$ of the fixed length as for any $i,j, w_1,\cdots,w_{n-1} \in \Sigma$
\begin{equation}
    \mu\left([i, w_1,\cdots,w_{n-1},j]\right) = \displaystyle \frac{p_i~ v_j}{\rho^n}. \label{eqn:mu_cylinder2}
\end{equation}
Let $\rr \in M_d$ be a  rational vector with positive entries and $\mathcal{C} \subset M_d$ a small neighborhood of $\rr.$ Let 
$$B_n = \{w \in \Delta_{\A} : Z_n(w) \in \mathcal{C}\}$$
From LDP, we know that $\frac{1}{n}\log\mu(B_n)$ is close to $-I(\rr)$ when $n$ is large. On the other hand, from \eqref{eqn:mu_cylinder2} we get that 
$\frac{1}{n}\log\mu(B_n)$ is close to $\frac{1}{n}\log\left(\rho^{-n}\cdot l_{n\rr}\right),$ where
$$L_1(\z) = \displaystyle\sum_{\ii \in \N^d} l_{\ii}\z^{\ii}$$ is a multivariate growth series of $L_1.$ Hence, in the limit when $n\rightarrow \infty$ we get the equality
\begin{eqnarray}
     I(\rr) &=& \log\rho - \displaystyle \limsup_{n\rightarrow\infty} \frac{1}{n} \log l_{n\rr} \nonumber\\
     &=& \log\rho - \psi(\rr) \label{eqn:I_psi_L1}
\end{eqnarray}
In fact, under the impose conditions using the definition of $\psi(\rr)$ one can prove \eqref{eqn:I_psi_L1} for all $\rr \in M_d,$ not just rationals. We summarize this as 
\begin{Pro}\label{pro:I_psi_L}
The indicatrice of growth $\psi(\rr)$ of a language determined by ergodic automaton $\A$ of type described above satisfies 
$$I(\rr) = \log\rho - \psi(\rr),$$ 
where $I(\rr)$ is the rate function associated with the Markov chain determined by the stochastic matrix $P$ corresponding to $A$ and  
\begin{equation}
    I(\rr) = \displaystyle \sup_{\uu \gg \oo} \sum_{j=1}^d r_j \log \left[\frac{u_j}{(\uu P)_j}\right],  ~~~\rr \in M_d. \label{eqn:30*}
\end{equation}
\end{Pro}
Finally, we make one more remark. Recall that entries of $A$ belong to the set $\{0,1\}.$ Let as before $\vv = (v_1,\cdots,v_d)$ be a right eigenvector of $A$ corresponding to $\rho=\rho(A), A\vv =\rho\vv.$ Assume that $\vv$ is a probability vector. Let $$M_d^* = \displaystyle \left\{\rr = (r_1,\cdots,r_d) \in M_d: r_i > 0, ~\forall~ i \right\}.$$ Then the map $T:M_d^* \rightarrow \R^d$ given by 
$$T(\q) = \s = (s_1,\cdots,s_d),$$ where 
\begin{equation}
    s_j = \displaystyle \frac{q_j}{v_j \sum_{i=1}^d \frac{q_i a_{ij}}{v_i}}, j = 1,\cdots,d. \label{eqn:T_S}
\end{equation}
Let $A(\z)$ be the matrix from Proposition \ref{pro:multigrow} i.e. matrix obtained from $A$ by replacing each $1$ in the $j$th column of $A$ by $z_j$ and let $\tv$ be a vector obtained from $\q \in M_d^*$ by 
$$\tv = \left(\frac{q_1}{v_1},\cdots,\frac{q_d}{v_d}\right)$$
\begin{lem}\label{lem:tAs_t}
For each $\q \in M_d^*$ vector $\tv$ satisfies $\tv A(\s) =\tv,$ where $\s = (s_1,\cdots,s_d)$ is given by \eqref{eqn:T_S}.
\end{lem}
\begin{proof} For $1\leq j \leq d,$
\begin{eqnarray}
\left(\tv A(\s)\right)_j &=& s_j \displaystyle \sum_{i=1}^d \frac{q_i a_{ij}}{v_i} \nonumber\\
   &=& \displaystyle \frac{q_j}{v_j\sum_{i=1}^d \frac{q_i a_{ij}}{v_i}} \cdot \sum_{i=1}^d \frac{q_i a_{ij}}{v_i} \nonumber\\
   &=& \displaystyle \frac{q_j}{v_j} = t_j
\end{eqnarray}
 (in the above relations we used the fact that $a_{ij} \in \{0,1\}$).
\end{proof}
Observe that $\s$ depends on the vector $\q$, so we can write $\s = \s(\q).$   
\begin{cor}\label{cor:detA_S}
$$\det\left(I - A(\s(\q))\right) = 0, ~~\forall~\q \in M_d^*.$$
\end{cor}
Recall that the multivariate growth series of language $L$ satisfies Proposition \ref{pro:multigrow}. i.e. 
$$\Gamma_L(\z) = \displaystyle \frac{G(\z)}{\det\left(I - A(\z)\right)}$$
and singularities of $\Gamma_L(\z)$ are determined by the roots of denominator. 
Hence, $T(M_d^*)$ is a part of the set of the real singularities of $\Gamma_L(\z)$ and is part of the boundary $\partial\D$.\\

\begin{figure}[!htb]
\begin{subfigure}[b]{0.4\textwidth}
\centering
    \includegraphics[width=\textwidth]{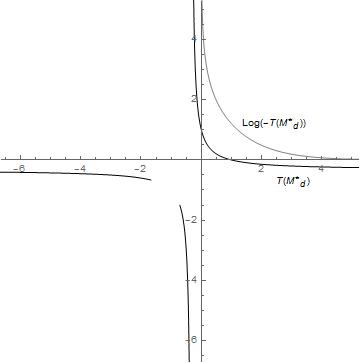}
    \caption{}
   \label{subfig:free T}
\end{subfigure}
\hfill
\begin{subfigure}[b]{0.6\textwidth}
\centering
    \includegraphics[width=\textwidth]{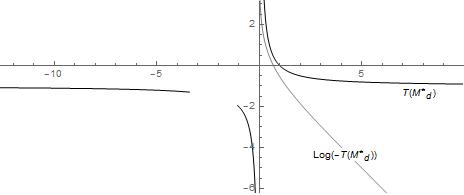}
    \caption{}
    \label{subfig:Fib T}
\end{subfigure}
\caption{The graphics of the sets $T(M_d^*)$ and $\log(-T(M_d^*))$ presented in \eqref{subfig:free T} and \eqref{subfig:Fib T} are associated with languages $L_{F_2}$ and $L_{Fib}$, respectively.}
\label{fig:set_T}
\end{figure}

We already know from Proposition \ref{pro:I_psi_L} that
$$I(\rr) = \log\rho - \psi(\rr).$$
Recall that $P = (p_{ij})$ and $p_{ij} = \displaystyle \frac{v_j a_{ij}}{\rho v_i}$. Rewriting \eqref{eqn:30*} as
\begin{eqnarray}
I(\rr) &=& \displaystyle \sup_{\q \in M_d^* } \left\{\sum_{j=1}^d r_j\log\rho - \displaystyle\sum_{j=1}^d r_j \log \left(
\frac{v_j \sum_{i=1}^d \frac{q_ia_{ij}}{v_i} }{q_j}\right)\right\} \nonumber\\
 &=& \log \rho - \displaystyle \inf_{\q \in M_d^*} \sum_{j=1}^d r_j (-1)\log \left[T(\q)\right]_j\nonumber\\
 &=& \log\rho - \displaystyle \inf_{\s \in T(M_d)} \left(-\sum_{j=1}^d r_j \log s_j\right)\nonumber
\end{eqnarray}
we conclude
$$ \psi(\rr) = \displaystyle \inf_{\s \in T(M_d)}  \left(-\sum_{j=1}^d r_j \log s_j\right)$$
Comparing with \eqref{eqn:psi_as_inf_on_dabaD} we observe that in the described situation, the infimum in \eqref{eqn:psi_as_inf_on_dabaD} can be taken only via the subset of the positive part of the boundary $\partial\D$. See Figures \ref{fig:set_T}.
 
\section{Finer asymptotic}\label{sec:alter}
As was already mentioned in the introduction, in the case of rational vector $\rr$ the function $\psi(\rr)$ can be identified by \eqref{eqn:multigrow_rational} (alternatively, the function $\varphi(\rr) = e^{\psi(\rr)}$ can be identified by \eqref{eqn:varphi_for_rational_r}).
A powerful results of the theory of ACSV (asymptotic combinatorics in several variables) presented in \cite{melczer2020invitation} allow not only to compute in many cases $\varphi(\rr),$ for rational functions $\Gamma(\z) = \displaystyle\frac{G(\z)}{H(\z)}$ and rational vectors $\rr$ but also to describe much finer asymptotics of the
diagonal coefficients $f_{n\rr}$ (See Theorems 5.1,5.2,5.3 in \cite{melczer2020invitation}). Also there are statements on the smoothness of $\varphi(\rr)$ as a function of rational $\rr.$\\

Our definition of $\varphi(\rr),$ that follow the idea from \cite{quint2002divergence} and is based on the cone approach works for arbitrary direction $\rr \in M_d.$ Moreover, under the assumption (CG) in many ``good''examples (including considered in this paper) the upper bound \eqref{eqn:varphi_interms_hr} can be replaced by the equality 
$$ \varphi(\rr) =  e^{ ~\displaystyle \inf_{\x \in \overline{\D}}h_{\rr}(\x)} = e^{~ \displaystyle\inf_{\x \in \partial\overline{\D}}h_{\rr}(\x)} = e^{\psi(\rr)} $$
when $\rr$ is rational, where $$\psi(\rr) = \displaystyle \inf_{\theta \in  \partial\left(-\overline{\Omega}\right)} \langle \rr, \theta\rangle,~~ \Omega = \relog(\D)$$
Additionally, the facts based on the convex analysis and Large Deviation Theory allow to claim that for the good rational functions $\Gamma(\z)$, the function $\varphi(\rr)$ is a real analytic function. See \cite{GQ23}.\\

Now we recall few definitions and results presented in \cite{melczer2020invitation}, apply them to our examples and make a comparison. The Theorem 5.1 from \cite{melczer2020invitation} basically states the following.\\

Let $\rr \in \Q^d$ and let $G(\z), H(\z) \in \Q[\z]$ be coprime polynomials such that $$\Gamma(\z) = \displaystyle \frac{G(\z)}{H(\z)}$$ admits a power series expansion $$\Gamma(\z) = \displaystyle \sum_{\ii \in \N^d} f_{\ii}\z^{\ii}.$$ Suppose that the system of equations
\begin{equation}
    H(\z) = r_2z_1H_{z_1}(\z)-r_1z_2H_{z_2}(\z)= r_dz_1H_{z_1}(\z)-r_1z_dH_{z_d}(\z) \label{sys:mel_5.1}
\end{equation}
admits a finite number of solutions, exactly one of which $\z' \in \Co_*^d$ is minimal (i.e. no other singularity $\z$ of $\Gamma(\z)$ satisfies $|z_j|<|z_j'|$ for all $1\leq j \leq d$).\\

Suppose that $H_{z_d}(\z) \neq 0, G(\z') \neq 0.$ Then as $n \rightarrow \infty$
\begin{equation}
    f_{n\rr} = \z'^{-n\rr}n^{\frac{1-d}{2}}\displaystyle\frac{(2\pi r_d)^{\frac{1-d}{2}} -G(\z')}{z_d'H_{z_d}(\z')\sqrt{\det(\mathcal{H})}}\left(1+O\left(\frac{1}{n}\right)\right) \label{eqn:5.1_mel}
\end{equation}
when $n\rr\in \N^d,$ where $\mathcal{H}$ is a $(d-1)\times(d-1)$ matrix defined by Equation (5.25) in \cite{melczer2020invitation} and it is supposed that 
$\det(\mathcal{H}) \neq 0.$\\

It is also claimed that as $\rr$ varies in any sufficiently neighborhood in $\R^d_{>0},$ the solution $\z' = \z'(\rr)$ varies smoothly with $\rr.$ The factor $(\z')^{-n\rr}$ in \eqref{eqn:5.1_mel} gives us 
$$\displaystyle \limsup_{n\rightarrow\infty} |f_{n\rr}|^{\frac{1}{n}} = |(\z')^{-\rr}| = e^{-\sum r_i\log|z'_i|} = e^{-h_{\rr}(\z')}$$
The system of equations \eqref{sys:mel_5.1} is equivalent in our situation (assuming condition (CG)) to the system of equations coming from the Lagrange multipliers method because if
\begin{equation}
    \Phi(\z,\lambda) = \displaystyle\sum_{i=1}^d r_i\log z_i - \lambda H(\z) \label{eqn:lagrange}
\end{equation}
then the critical points are solutions of the system 
\begin{eqnarray}
\displaystyle\frac{\partial\Phi}{\partial z_i} &=& \displaystyle\frac{r_i}{z_i} - \lambda H_{z_i} = 0 , i = 1,\cdots,d \nonumber\\
H(\z) &=& 0 \label{sys:lagrange}
\end{eqnarray}
which is equivalent to \eqref{sys:mel_5.1}. In our situation, we make substitutions $z_i = e^{-\theta_i}$ and replace \eqref{eqn:lagrange} by the
$$ \Phi'(\theta,\lambda) = \displaystyle\sum_{i=1}^d r_i\theta_i - \lambda H(e^{-\theta}) $$
as shown in the previous sections. Of course \eqref{eqn:5.1_mel} is much finer asymptotic than 
$$ f_{n\rr}  \sim (\z')^{-n\rr}. $$
As was already mentioned, the condition (CG) gives an alternative definition of the growth in the direction of $\rr$ that works not only for rational $\rr$ but for any direction $\rr \in \R_{>0}^d.$ At the same time the above argument show that when (CG) holds and conditions of Theorem 5.1 in \cite{melczer2020invitation} are satisfied, then the rates of growth in the rational direction $\rr$ defined by \eqref{eqn:multigrow_rational} or as $e^{-\psi(\rr)}$ coincide. The smooth dependence of $\z' =\z'(\rr)$ (and hence of $\varphi(\rr) = \z'^{-\rr}$) on $\rr \in \Q_{>0}^d$ can be strengthen to the claim about the analytic dependence on $\rr$ for all $\rr \in \R^d_{>0},$ where $\psi(\rr) \geq 0$ as shown in \cite{GQ23}.\\

An useful tool in discussed topics is the logarithmic gradient map
$$ \displaystyle \triangledown_{\log}f = \left(z_1f_{z_1},\cdots,z_df_{z_d}\right).$$
Proposition 3.13 in \cite{melczer2020invitation} states that for any minimal singular point $\z'$ of $\Gamma(\z) = \displaystyle\frac{G(\z)}{H(\z)}$ (where $G,H$ are coprime) there exists $\rr \in \R^d_{\geq 0}$ and $\tau \in \Co$ such that 
$$ \displaystyle \left(\triangledown_{\log} H^{(\s)}\right)(\z') = \tau\cdot\rr,$$ where $H^{(\s)}$ is a square free part of $H.$ In this situation $\z'$ is either a minimizer or a maximizer of the height function $h_{\rr}(\z)$ on $\overline{\D}.$ \\ 

Let us apply Theorem 5.1 from \cite{melczer2020invitation} to the case of the free group of rank $2.$ We assume that $\rr = (p,1-p) $ is rational, $\z = (x,y)$. Recall that the multivariate growth series of free group $F_2$ is 
$$\Gamma_{F_2}(x,y) 
= \displaystyle\frac{(1+x)(1+y)}{1-x-y-3xy}=\displaystyle\frac{G(x,y)}{H(x,y)}$$
and that the singularities of $\Gamma_{F_2}(x,y)$ are the points
$$\z' = (x,y) = \left(\displaystyle \frac{3 p-2 + 2 \sqrt{3 p^2-3 p+1}}{3 p}, \displaystyle\frac{1-3 p+2 \sqrt{3 p^2-3 p+1}}{3 (1-p)}\right)$$
whose coordinates are real numbers with positive coordinates.
Hence,
$$ (\z')^{-n\cdot\rr} = \left(\displaystyle \frac{2-3 p+ 2 \sqrt{3 p^2-3 p+1}}{p}, \displaystyle\frac{3 p-1+2 \sqrt{3 p^2-3 p+1}}{(1-p)}\right)^{n\cdot\rr} = e^{n\cdot\psi_{F_2}(\rr)}$$
We quickly check that the assumption of Theorem 5.1 from \cite{melczer2020invitation} are satisfied. In other words, we need to check that the partial derivative $\frac{\partial H}{\partial y}$ does not vanish at $\z'$ and that the matrix 
$\mathscr H$ from Equation (5.25) in \cite{melczer2020invitation} is non singular (with ${\bf w}=\z'$). Indeed, a direct computation gives 
$$\frac{\partial H}{\partial y}(x,y)=-1-3x$$
which is non zero at $\z'$. Now, the dimension $d$ being two, still with the notation of \cite{melczer2020invitation},
the matrix $\mathscr H$ is the scalar
\begin{align*}
\mathscr H&=
V_1 + V_1^2+U_{1,1} - 2V_1U_{1,2} + V^2_1U_{2,2}\\
&=\frac{x(1+3y)}{y(1+3x)}+\left(\frac{x(1+3y)}{y(1+3x)}\right)^2
+0-2\frac{x(1+3y)}{y(1+3x)}\frac{3xy}{y(1+3x)}+0\\
&=\frac{xy+3x^2y+3xy^2+x^2}{y^2(1+3x)^2}>0.
\end{align*}
Therefore, following (\cite{melczer2020invitation}, Equation 5.1) we get for rational $\rr$ 
\begin{equation}
    f_{n\cdot\rr} = ce^{n\cdot\psi_{F_2}(\rr)}n^{-\frac{1}{2}}\left(1+O\left(\frac{1}{n}\right)\right) \label{eqn:gq23}
\end{equation}
where $c = c(p)$ does not depend on $p$. In fact the results of \cite{GQ23} allows to have relation \eqref{eqn:gq23} when $\rr$ is irrational, only the left hand side should be replaced by the sum of the coefficients $\gamma_{\ii}$ in the uniformly bounded neighborhood of the point $n\rr.$ 

\section{Concluding remarks and open questions}\label{sec:open}
Finding of $\psi_{F_m}(\rr),$ where $\rr = (p,q,1-p-q)$ for free group $F_3$ of rank $3$ leads to the solving of polynomial equation of degree $4$ in variable $z = e^s$
\begin{multline}\label{eqn:deg_4}
3p^2z^4 + 4p(7p-2)z^3 + 2\left(33p^2-32pq-8p-32q^2+32q-8\right)z^2 \\+ 12p(5p-6)z - 45p^2=0
\end{multline}
and it can be solved in radicals. Substituting $p = q = \frac{1}{3}$ in \eqref{eqn:deg_4} we obtain that $s = \log 5$ and hence we get a value $\psi_{F_3}(\frac{1}{3},\frac{1}{3},\frac{1}{3}) = \log 5.$ So the multivariate growth in this case coincides with the ordinary growth and $\log 5$ is a maximal value of $\psi(\rr)$. The higher ranks $m = 4,5,\cdots$ lead to polynomial equations of degree $> 5$ and most probably obtaining of the precise analytic expressions for $\psi_{F_m}(\rr)$ is impossible. But at least we know that $\psi_{F_m}(\rr)$ is a real analytic concave function \cite{GQ23} with a maximum value $\log(2m-1)$ achieved at unique point $\rr = \left(\frac{1}{m},\cdots,\frac{1}{m}\right).$\\

Now, let us go back to cogrowth. It can be shown that the condition (CG) always hold for a subgroup $H < F_m$ and so the formula \eqref{eqn:psi_inf1} is applicable. If $H < F_m$ is a finitely generated subgroup then it is represented by a regular language \cite{DGS2021} and hence, its cogrowth and multivariate cogrowth series are rational.  The Conjecture claiming that $\Gamma_H(z), z \in \Co$ is rational if and only if $H$ is finitely generated was stated in \cite{DGS2021} and it is known that this conjecture is true in the case of normal subgroups. A similar conjecture can be stated for multivariate cogrowth series $\Gamma_H(\z), \z \in \Co^{m}.$
Also, we state the following:   
{\conj \label{conj:cogrowth_general} Let $N \vartriangleleft F_m.$ Then $F_m/N$ is amenable if and only if $\psi_{N}(\rr) = \psi_{F_m}(\rr).$}\\

By cogrowth criteria of amenability we know that in the case when $F_m/N$ is amenable, the relation
\begin{equation}
     \log(2m-1) = \max_{\rr} \psi_{F_m}(\rr) =  \max_{\rr} \psi_N(\rr) \label{eqn:log_Fm_H}
\end{equation}
hold.
It is unclear if $\psi_N(\rr)$ may have values less than the values of $ \psi_{F_m}(\rr)$ in the case when the Equation \eqref{eqn:log_Fm_H} holds. Even the case when $N = [F_2,F_2]$ is a commutator subgroup of $F_2$ deserves a separate consideration.\\

And finally, there is a formula in \cite{akemann_1976_MR442698} 
$$\chi(p) = 2 \displaystyle \min_t \left[ \sum_{i=1}^m \sqrt{t^2+p_i^2} - (m-1)t\right]$$
for the spectral radius $\chi(p)$ of a symmetric random walk on a free group $F_m$ given by a positive vector $p = (p_1,\cdots,p_m), 2\sum p_i = 1$ where $p(a_i) = p(a_i^{-1}) = p_i$. \\

Computation  of  $\chi(p)$  in  the  case  of  rank  2  leads  to  the equation  of  degree 4 in variable $x = t^2$
\begin{eqnarray}
    3 x^4 + 4\left(p_1^2 + p_2^2 \right) x^3 + 6 p_1^2 p_2^2 x^2 - 
 p_1^4 p_2^4 =0 \label{eqn:deg_8}
\end{eqnarray}
and hence $\chi(p)$ can be expressed in radicals. Taking $p_1=p_2 = \frac{1}{4}$, \eqref{eqn:deg_8} leads  to  the  equation
$$(1 + 16 x)^3 (-1 + 48 x) = 0$$
which  gives  a  value $\chi(p) = \displaystyle\frac{\sqrt{3}}{2}$.\\

The latter number is known since 1959 due to H. Kesten \cite{kesten1959random} who, in particular proved that for a simple random walk on $F_m$ the spectral radius $\chi = \displaystyle \frac{\sqrt{2m-1}}{m}.$ Higher rank leads to solving polynomial equations of degree $> 5$ and expressing $\chi(p)$ in radicals seems to be impossible for $F_m, m \geq 3$ and arbitrary $p.$\\

Let $H < F_m$ and $\chi_{F_m/H}(p)$ be a spectral radius of a random walk on a Schreier graph $\Lambda = \Lambda(F_m, H,\Sigma)$ given by probabilities $p_i, 1\leq i \leq m$.  We end up with the following question. 
\begin{que}\label{que:chi_varphi}
Is there a formula expressing $\chi_{F_m/H}$ via $\alpha_H(p)$, where $\alpha_H(p) = \varphi_H(p)$ is a multivariate growth of $\Delta_H(2p)$ in the direction prescribed by the vector $2p \in M_m$ ? Does such a formula exists when $H$ is normal subgroup in $F_m$ and hence $\Lambda = \Lambda(F_m, H,\Sigma)$ is a Cayley graph. 
\end{que}
\section{Acknowledgments} 
The first and the second author thank University of Geneva (Swiss NSF  grant 200020-200400). The third author acknowledges United States-India Educational Foundation(USIEF), New Delhi and U.S. Department of State for the Fulbright Nehru Postdoctoral Research Fellowship, Award No.2479/FNPDR/2019. Also he is grateful to S. P. Mandali, Pune and Government of Maharashtra for the approval of the study leave to
undertake the fellowship. Finally, he acknowledges the generous hospitality of Texas A \& M University, College Station, TX, USA during his visit when a big part of the project was realized.

\end{document}